\documentclass[11pt,a4paper,twoside]{article}

\usepackage{iftex}
\usepackage[T1]{fontenc}
\ifPDFTeX
  \usepackage[utf8]{inputenc}
\fi
\usepackage{lmodern}

\usepackage{amsmath,amsthm,amssymb,mathtools}

\usepackage{geometry}
\usepackage{xcolor}
\usepackage{microtype}
\usepackage{fancyhdr}

\usepackage[
  colorlinks=true,
  linkcolor=blue!45!black,
  citecolor=blue!45!black,
  urlcolor=blue!45!black
]{hyperref}

\allowdisplaybreaks[3]
\newtheoremstyle{paperplain}{0.65em}{0.65em}{\itshape}{}%
  {\bfseries}{.}{0.5em}{}
\theoremstyle{paperplain}
\newtheorem{theorem}{Theorem}[section]
\newtheorem{proposition}[theorem]{Proposition}
\newtheorem{lemma}[theorem]{Lemma}
\newtheorem{corollary}[theorem]{Corollary}
\newtheoremstyle{paperremark}{0.65em}{0.65em}{\normalfont}{}%
  {\bfseries}{.}{0.5em}{}
\theoremstyle{paperremark}
\newtheorem{remark}[theorem]{Remark}

\newcommand{\runningtitle}{}
\let\originaltitle\title
\renewcommand{\title}[2][]{%
  \originaltitle{#2}%
  \if\relax\detokenize{#1}\relax
    \gdef\runningtitle{#2}%
  \else
    \gdef\runningtitle{#1}%
  \fi
}

\title[Graphical scalar-curvature equation]{\textbf{ Interior curvature estimates for the graphical
scalar curvature equation in all dimensions}}

\author{%
  \textbf{Guohuan Qiu}
  \and
  \textbf{Jin Yan}
}

\date{}
\makeatletter
\newcommand{\printfirstpagebottommatter}{%
  \begingroup
  \renewcommand{\thefootnote}{}%
  \renewcommand{\@makefntext}[1]{\noindent##1}%
  \footnotetext[0]{%
    \textit{Keywords.}
    Scalar curvature equation; interior curvature estimate;
    doubling method; two-surface comparison; Pogorelov estimate.
    \par\smallskip
    \textit{2020 Mathematics Subject Classification.}
    35J60, 35B45, 53C42.%
  }%
  \endgroup
}
\makeatother
\newcommand{\printauthorinformation}{%
  \par\bigskip
  \begingroup
  \small
  \setlength{\parindent}{0pt}%

  \noindent
  \textbf{Guohuan Qiu}\par
  State Key Laboratory of Mathematical Sciences,
  Academy of Mathematics and Systems Science,
  Chinese Academy of Sciences,
  Beijing 100190, China; \\
  Institute of Mathematics,
  Academy of Mathematics and Systems Science,
  Chinese Academy of Sciences,
  55 Zhongguancun East Road, Beijing 100190, China.\par
  \textit{Email:}
  \href{mailto:qiugh@amss.ac.cn}{qiugh@amss.ac.cn}.

  \par\medskip

  \noindent
  \textbf{Jin Yan}\par
  Institute of Mathematics,
  Academy of Mathematics and Systems Science,
  Chinese Academy of Sciences,
  55 Zhongguancun East Road, Beijing 100190, China.\par
  \textit{Email:}
  \href{mailto:yanjin@amss.ac.cn}{yanjin@amss.ac.cn}.

  \endgroup
}

\begin{document}
\maketitle
\printfirstpagebottommatter
\begin{abstract}
We establish interior curvature estimates for admissible solutions
of the constant graphical scalar curvature equation in every
dimension, thereby settling an open problem in fully nonlinear
elliptic PDEs. The proof combines Jacobi inequalities and a modified
auxiliary function with a two-surface maximum principle and
a corresponding Pogorelov estimate.
\end{abstract}

\section{Introduction}\label{sec:introduction}

Let $X(x)=(x,u(x))$ be a hypersurface graph in $\mathbb R^{n+1}$.
We consider the scalar curvature equation
\begin{equation}\label{eq:intro-equation}
 \sigma_2(\kappa[u]):=\frac{H^2-|A|^2}{2}=1,
 \qquad \kappa[u]\in\Gamma_2,
\end{equation}
where $\kappa[u]$, $H$ and $A$ denote the principal curvatures,
mean curvature and second fundamental form with respect to the
downward unit normal, and
\[
 \Gamma_2
 :=
 \left\{
 \kappa\in\mathbb R^n:
 \sum_i\kappa_i>0,\quad
 \sum_{i<j}\kappa_i\kappa_j>0
 \right\}
\]
is the G{\aa}rding cone. This fully nonlinear equation is elliptic on $\Gamma_2$.
Under a fixed gradient bound, an a priori curvature bound
yields uniform ellipticity.

Our study of the scalar curvature equations
is motivated by their connections with classical Monge--Amp\`ere
theory and differential geometry. In dimension two,
$\sigma_2(D^2u)=\det D^2u$, while $\sigma_2(\kappa)$ is the Gauss
curvature. Both equations therefore have a Monge--Amp\`ere structure.
Early a priori estimates for elliptic Monge--Amp\`ere equations
were obtained by Lewy \cite{Lewy1935,Lewy1937}.
Such equations arise in the Weyl isometric embedding problem
\cite{Alexandrov1948,Heinz,Lewy1938Weyl,Nirenberg1953,Pogorelov1973,Weyl1916}, the classical Minkowski problem in convex geometry
\cite{Alexandrov1938,ChengYau1976,Lewy1938Minkowski,Minkowski1903,
Nirenberg1953,Pogorelov1978}, and optimal transport
with quadratic cost \cite{Brenier1991,Caffarelli1992}. In dimension three, on the $\Gamma_2$ branch, the equation
$\sigma_2(D^2u)=1$ is equivalent to the special Lagrangian
equation at the critical phase $\pi/2$, arising from the
special Lagrangian geometry of Harvey and Lawson
\cite{HarveyLawson}. 
For Euclidean hypersurfaces, the Gauss equation relates $\sigma_2(\kappa)$
to scalar curvature, a fundamental intrinsic invariant:
\[
 \operatorname{Scal}=2\sigma_2(\kappa).
\]

The existence and boundary regularity theory for Hessian and
curvature equations has been developed through the study of
Dirichlet and Neumann problems. Caffarelli, Nirenberg and Spruck \cite{CNSHessian,CNS}
established fundamental results for the Dirichlet problems for
$\sigma_k$-Hessian and $\sigma_k$-curvature equations, and Ivochkina \cite{Ivochkina1990,Ivochkina} obtained further solvability
results for prescribed curvature equations. For Neumann problems, Lions,
Trudinger and Urbas \cite{LionsTrudingerUrbas} established a priori estimates and classical
solvability for Monge--Amp\`ere type equations. Ma and the first
author \cite{MaQiuNeumann} resolved a conjecture of Trudinger by establishing
global second-derivative estimates and classical solvability
for the Neumann problem for $k$-Hessian equations on smooth
uniformly convex domains. More recently, the first
author \cite{qiu2026robin} obtained boundary
second-derivative estimates for admissible scalar curvature
graphs in dimension three.

Purely interior second-derivative estimates are more difficult
to obtain, as they depend only on lower-order norms and
the equation data, without relying on the regularity of the
boundary data. Whether such estimates hold depends crucially on the value of $k$.
In dimension two, Heinz \cite{Heinz1956,Heinz} established the classical interior
$C^2$ estimate for the Monge--Amp\`ere equation.
For alternative proofs, see \cite{ChenHanOu,Liu}. In dimensions $n\geq3$, Pogorelov \cite{Pogorelov1978}
constructed singular convex solutions of the Monge--Amp\`ere
equation. Urbas \cite{Urbas} extended these constructions
to $\sigma_k$-Hessian and $\sigma_k$-curvature equations for
$3\leq k\leq n$. These examples show that general
interior second-derivative estimates fail for $k\geq3$,
highlighting the special role of the quadratic case $k=2$.

For the quadratic Hessian equation
\[
                         \sigma_2(D^2u)=1,
\]
Warren and Yuan \cite{WarrenYuan} achieved the first breakthrough beyond dimension
two by exploiting the minimal-surface structure of the associated
critical-phase special Lagrangian equation to establish the
interior estimate in dimension three.
In general dimensions, Guan and the first author \cite{GuanQiu} obtained
interior estimates under a fixed lower bound for $\sigma_3(D^2u)$. 
Their proof introduced $x\cdot Du-u$ as an auxiliary function
in the maximum-principle approach to these interior estimates.
This function subsequently became a key ingredient in the
first author's doubling argument and its later developments.
Convex solutions were also treated independently by McGonagle, Song and Yuan
\cite{McGonagleSongYuan} and Mooney \cite{Mooney}.
For variable right-hand sides, the first author \cite{QiuHessian} established
interior estimates in dimension three for admissible solutions
with general positive $C^2$ data $f(x,u,Du)$.
The proof developed Guan--Qiu's auxiliary-function construction into
a doubling argument that propagates Hessian bounds from
a smaller ball to a larger region.
Shankar and Yuan \cite{ShankarYuanSemi} later established interior estimates for
semiconvex solutions in arbitrary dimensions. In a major subsequent advance \cite{ShankarYuan}, they resolved the problem in dimension four.
Building on the first author's maximum-principle doubling
argument \cite{QiuHessian}, they established an almost Jacobi
inequality whose gradient coefficient may degenerate.
They overcame this difficulty by exploiting the favorable
ellipticity of the corresponding Hessian configurations
and a compensating $|Du|^2$ term in the test function.
This yielded the doubling estimate needed in dimension four.
They also adapted the argument of Chaudhuri and Trudinger
\cite{ChaudhuriTrudinger2005} to establish almost-everywhere
second differentiability of viscosity solutions.
Savin's small perturbation theorem \cite{Savin2007} then
provided local Hessian control, which the doubling estimate
propagated within a compactness argument.
Their method also yielded higher-dimensional estimates
under a dynamic semiconvexity condition. Fan \cite{FanHessian} subsequently extended these estimates to equations with
general positive $C^{1,1}$ right-hand sides $f(x,u,Du)$.
Most recently, Li and Wu resolved the interior regularity problem
for the $\sigma_2$-Hessian equation in every dimension through
a new nonlinear comparison method \cite{LiWu}.

For the graphical scalar curvature equation
\eqref{eq:intro-equation}, progress has been more limited.
Guan and the first author established interior curvature
estimates for convex solutions \cite{GuanQiu}.
In dimension three, the first author \cite{Qiu3D} obtained the unrestricted
estimate by adapting Warren and Yuan's integral approach
to the Gauss graph $(X,\nu)$, viewed as a submanifold of higher
codimension with bounded mean curvature. For convex solutions, Chen, Jian and Zhou \cite{ChenJianZhou2026} obtained interior
estimates for the Lipschitz positive right-hand side.
More recently, Fan and Shankar \cite{FanShankar} established an interior curvature estimate in dimension four whose constant also depends on the modulus of continuity of $Du$. In adapting the first author's maximum-principle doubling
argument \cite{QiuHessian} to the curvature equation, Fan and
Shankar encountered an additional gradient-dependent term,
which they controlled under a local small-gradient condition
\cite{FanShankar}. Our modified auxiliary function controls
this term under an arbitrary fixed gradient bound. A separate difficulty is
that the comparison graphs have different metrics and normals.
We address this mismatch by comparing paired points on the
two graphs. Together with a two-surface Pogorelov estimate,
these ideas yield the following quantitative curvature
estimate, with a constant depending only on the dimension
and the height and gradient bounds.

\begin{theorem}\label{thm:main}
Let \(n\geq3\).  Suppose that \(u\in C^\infty(B_{3/2})\) satisfies
\begin{equation}\label{eq:main-assumption}
 \kappa[u]\in\Gamma_2,
 \qquad
 \sigma_2(\kappa[u])=1\quad\text{in }B_{3/2},
 \qquad
 \|u\|_{L^\infty(B_{3/2})}+\|Du\|_{L^\infty(B_{3/2})}\leq K.
\end{equation}
Then
\begin{equation}\label{eq:main-estimate}
\sup_{B_{1/2}}|\kappa[u]|\leq C(n,K).
\end{equation}
\end{theorem}
\begin{remark}
For simplicity, we consider only the constant right-hand side
$f\equiv1$. Our method can be extended to
$\sigma_2(\kappa[u])=f(X,\nu)>0$ with smooth $f$, see \cite{QiuYangradient}.
As in the first author's three-dimensional estimate \cite{Qiu3D},
the expected constants depend on the dimension, the $C^1$ bound
for the graph, $\|f\|_{C^2}$, and $\|f^{-1}\|_{L^\infty}$.
Chen, Zhou and Zhu recently established the corresponding
Schauder regularity for the quadratic Hessian equation on the
full $\Gamma_2$ branch \cite{ChenZhouZhu2026}.
The analogous problem for the graphical scalar curvature equation
is not resolved by the present argument.
\end{remark}

We first recall the idea for the $\sigma_2$-Hessian equation. Mooney observed in
\cite[Remark~2.6]{MooneyRemarks} that the Chou--Wang Pogorelov
estimate requires only the one-sided condition $L_u (v-u)\geq -1$
for a comparison function $v>u$ agreeing with $u$ on the
boundary, where $L_u$ denotes the linearization of
$\sigma_2^{1/2}$ at $u$. This suggests constructing $v$ by
solving the linearized equation $L_uv=0$ with boundary data
$u$. To obtain an interior Hessian estimate, one then needs
quantitative control of the gap $v-u$.

Li and Wu's key new idea is to replace the linearized Dirichlet
problem by a nonlinear comparison problem. Given a solution
of $\sigma_2(D^2u)=1$, they solve
\[
 \sigma_2(D^2v)=\frac14\quad\text{in }B_{3/2},
 \qquad
 v=u\quad\text{on }\partial B_{3/2}.
\]
This choice makes both Trudinger's interior gradient control \cite{trudinger1997weak} and the
required one-sided linearized inequalities available.
Indeed, writing $L_v$ for the linearization of
$\sigma_2^{1/2}$ at $v$, concavity and homogeneity give
\[
 L_u(v-u)\geq-\frac12,
 \qquad
 L_v(v-u)\leq-\frac12.
\]
With these gradient bounds and linearized inequalities,
the Chou--Wang  Pogorelov estimate \cite{ChouWang}  reduces the problem to
establishing uniform separation of $v$ from $u$ on $B_{1/2}$. Li and Wu obtained this separation by first establishing a
positive lower bound on a small interior ball through a localized
integration-by-parts argument, and then propagating it to
$B_{1/2}$ using exponential barriers along a finite chain of
balls \cite{LiWu}. The propagation step uses the same basic
maximum-principle mechanism as the first author's doubling
argument \cite{GuanQiu,QiuHessian}, formulated through exponential
barrier comparisons. In the present paper, we adapt the first
author's direct maximum-principle argument to propagate separation
for the graphical curvature equation. Appendix~\ref{appendix}
gives the corresponding direct proof for the $\sigma_2$-Hessian
equation. In both settings, a single auxiliary function propagates
the lower bound from a small interior ball to $B_{1/2}$, without
iteration along a chain of balls; see
Theorems~\ref{thm:separation} and~\ref{thm:single-maximum}.

We adapt this nonlinear comparison idea to
\eqref{eq:intro-equation}.  On \(B_1\) we solve
\begin{equation}\label{eq:intro-auxiliary}
 \sigma_2(\kappa[v])=\frac14,
 \qquad
 v=u\quad\text{on }\partial B_1.
\end{equation}
Ivochkina's Dirichlet theory \cite{Ivochkina} provides a smooth admissible
solution. The curvature normalization and
boundary compatibility conditions are verified in the proof
of Theorem~\ref{thm:main}. Comparison gives uniqueness and
$v>u$ in $B_1$. Korevaar's interior gradient estimate \cite{Korevaar} then
bounds the gradient of $v$ on a fixed interior subdomain,
with a bound depending only on $n$ and $K$.

The Li--Wu argument does not extend directly to the curvature
equation. The graphical curvature operator depends on both
$D^2u$ and $Du$. At the same base point, the two graphs have
different induced metrics and normals. The resulting curvature terms have no fixed sign. In fact,
Remark~\ref{rem:intrinsic-one-point-obstruction} gives admissible jets with
the two prescribed right-hand sides and bounded gradients for which both
endpoint linearizations of \(v-u\) have arbitrarily large values of the wrong
sign. Thus the endpoint linearization argument used in the Hessian
setting does not directly yield the required Jacobi inequality
for $-\log(v-u)$.

We overcome this obstruction by replacing the vertical gap with a geometric
one.  Let $d(y)$ be the Euclidean distance from $(y,v(y))$
to the closed hypograph of $u$. Under the Lipschitz bounds, \(d\) is
quantitatively equivalent to \(v-u\).  If \(T_v\) denotes the first Newton
tensor of the upper graph, then the following inequality holds in the
two-surface upper-support sense at every interior closest pair:
\begin{equation}\label{eq:intro-distance-jacobi}
 T_v^{ij}\nabla^2_{ij}(-\log d)
 \geq
 \frac1{2(n-1)}
 T_v^{ij}\nabla_i(-\log d)\nabla_j(-\log d)
 +\frac1{4d}.
\end{equation}

The geometry behind \eqref{eq:intro-distance-jacobi} is that
of parallel hypersurfaces. With our downward-normal convention, the upper parallel
hypersurface to the lower graph is
\[
 X_d=X-dN,
\]
where $N$ is the downward unit normal of the lower graph.
Before the first focal point, its shape operator is
\[
 A_d=A(I-dA)^{-1}.
\]
This formula appears naturally in the two-surface
second-variation argument
\cite{Eschenburg1987,HeintzeKarcher1978}.
Fix a point on the upper graph and a nearest point on the
lower graph. Their tangent spaces need not agree.
The natural rotation identifies these spaces isometrically.
The mixed second variation also involves the orthogonal
projection between them, which may degenerate in one direction.
A rank-one contribution from the gradient term compensates
for this loss. The compression inequality shows that the
corrected matrix lies in $\Gamma_2$ and gives a quantitative
lower bound for its $\sigma_2$. The mixed G{\aa}rding inequality
then converts the strict curvature gap into the positive term
in \eqref{eq:intro-distance-jacobi}.

A smooth torsion comparison uses this positive term to give
a quantitative lower bound for $d$ at an interior point.
Since $d$ has a controlled Lipschitz constant, this lower
bound extends to a fixed ball. An adaptation of the maximum-principle argument from \cite{GuanQiu,QiuHessian} to $-\log d$
then propagates the bound to $B_{1/2}$.
Its test function combines the graphical support-function term
introduced in \cite{GuanQiu} with a gradient correction weighted
by the horizontal distance to the cutoff center. This distance
weight generates a positive quadratic curvature term that absorbs
the gradient-dependent error. This establishes uniform separation between the two graphs.

The separation is then used in a two-surface Pogorelov estimate.
At a product-space maximum, the natural rotation identifies
the two tangent spaces. A directional Jacobi inequality for
a suitably weighted mean curvature and the mixed G{\aa}rding
inequality control the second-order terms at the paired points.
The resulting weighted estimate gives the interior curvature
bound, because uniform separation keeps the cutoff uniformly
positive on $B_{1/2}$.

Two-point and product-space maximum principles have a long
history in nonlinear PDE and geometric analysis.
Korevaar and Kennington developed convexity maximum principles,
systematically treated by Kawohl, while Rosay and Rudin
established a related principle for the convexity of level sets
\cite{Kawohl1985,Kennington1985,Korevaar1983,RosayRudin1989}.
Andrews and Clutterbuck developed two-point estimates for
moduli of continuity and later obtained a sharp log-concavity
comparison that led to their proof of the fundamental gap
conjecture \cite{AndrewsClutterbuck2009,AndrewsClutterbuck2011}.
For curve shortening flow, Huisken established a chord--arc
comparison, which Andrews and Bryan sharpened to obtain
curvature bounds and a direct proof of Grayson's theorem
\cite{AndrewsBryan2011,Huisken1998}.
Andrews gave a direct two-point proof of noncollapsing
for mean-convex mean curvature flow \cite{Andrews2012}.
Andrews, Langford and McCoy extended this approach to
fully nonlinear curvature flows and gave a product-space
proof of the known containment principle
\cite{AndrewsLangfordMcCoy2013}.
Andrews and Langford obtained further noncollapsing estimates
under inverse-concavity by optimizing paired second-derivative
directions \cite{AndrewsLangford2016}.
In the static setting, Brendle used a curvature-weighted
two-point function to resolve the Lawson conjecture
\cite{BrendleLawson2013}.
He also combined two-point differential inequalities with
integral estimates to obtain an asymptotically sharp
inscribed-radius bound under mean-convex mean curvature
flow \cite{BrendleInscribed2015}.
More recently, Colding and Minicozzi derived local elliptic
and parabolic inequalities for distances between hypersurfaces
in minimal-surface and mean-curvature-flow settings, using
direct Hessian calculations and supporting barriers at
nonsmooth points \cite{ColdingMinicozzi2026}.

Our two-surface Pogorelov estimate differs from the one-point
estimates of Urbas, Chou--Wang, and Sheng--Urbas--Wang
\cite{ChouWang,ShengUrbasWang2004,Urbas2000}.
The curvature estimate itself uses second-order information
at paired points on two distinct hypersurfaces.

The paper is organized as follows.
Section~\ref{sec:preliminaries} records the geometric conventions,
the identities for the first Newton tensor, and the mixed
G{\aa}rding inequality. Section~\ref{sec:separation} establishes uniform separation between the comparison graphs. To this end,
Subsection~\ref{subsec:The distance function} compares the ambient
distance with the vertical gap and explains the obstruction to a
direct one-point comparison. The algebraic compression inequality is proved in Subsection~\ref{subsec:The compression inequality} and used in
Subsection~\ref{subsec:The two-surface inequality} to derive the
distance Jacobi inequality in the upper-support sense.
Subsection~\ref{Subsec:A positive lower bound at some point} then
obtains a positive lower bound at an interior point by a torsion
comparison. A modified auxiliary function and a doubling argument
propagate this bound and complete the separation estimate. Section~\ref{sec:pogorelov} establishes the comparison Pogorelov
estimate and applies it to the interior curvature problem.
More precisely, Subsection~\ref{subsec:jacobi} derives the directional
mean-curvature Jacobi inequality and the estimates needed to control
its unfavorable terms. Subsection~\ref{Subsec:The two-surface Pogorelov estimate} establishes the two-surface Pogorelov estimate through a product-space maximum principle. Subsection~\ref{subsec:main-proof} then proves boundary decay of the comparison gap, verifies the solvability of the auxiliary Dirichlet problem, and combines these results with uniform separation to complete the proof of Theorem~\ref{thm:main}. Finally, Appendix~\ref{appendix} gives a direct maximum-principle proof
of the corresponding doubling estimate for the \(2\)-Hessian equation.

\section{Preliminaries}\label{sec:preliminaries}
We first fix the geometric conventions used throughout the paper.
Let
\[
 X(x)=(x,u(x)),\qquad
 W=(1+|Du|^2)^{1/2},\qquad
 \nu=\frac{(Du,-1)}{W},
\]
and let
\[
 E_{n+1}=(0,\ldots,0,-1)
\]
be the downward vertical unit vector.
Thus $\nu$ is the downward unit normal. We use the sign convention
\[
    g_{ij}=\delta_{ij}+u_i u_j,
    \qquad
    h_{ij}=-\langle \overline\nabla_{X_i}X_j,\nu\rangle
           =\frac{u_{ij}}{W}.
\]
In particular, \(\overline\nabla_{X_i}\nu=h_i^jX_j\).  Covariant
derivatives with respect to \(g\) are denoted by subscripts; in particular,
\(h_{ijk}=\nabla_kh_{ij}\).  The Codazzi equations give
\(h_{ijk}=h_{ikj}\).

Set
\[
    H=g^{ij}h_{ij},\qquad
    T_{u}^{ij}=\frac{\partial\sigma _2}{\partial h_{ij}}
          =Hg^{ij}-h^{ij}.
\]
We always work on the admissible branch \(\kappa[u]\in\Gamma _2\).  Hence
\((T_u^{ij})\) is positive definite.

The first Newton tensor is divergence-free and satisfies
\begin{equation}
    \nabla_i T_u^{ij}=0,
    \qquad
    T_u^{ij}g_{ij}=(n-1)H,
    \qquad
    T_u^{ij}h_{ij}=2 \sigma_2.
    \label{tor:newton-identities}
\end{equation}
We write
\[
    \Delta_{T_u} \phi=T_u^{ij}\nabla^2_{ij}\phi,
    \qquad
    |\nabla\phi|_{T_u}^2=T_u^{ij}\phi_i\phi_j,
    \qquad
    \langle\nabla\phi,\nabla\psi\rangle_{T_u}
       =T_u^{ij}\phi_i\psi_j.
\]

We shall also use the following form of the G{\aa}rding
inequality from \cite{Garding1959}:
\begin{equation}\label{eq:mixed-garding-prelim}
 T_1(A):B
 \geq 2\sqrt{\sigma_2(A)\sigma_2(B)},
 \qquad A,B\in\Gamma_2,
\end{equation}
where \(T_1(A)=\sigma_1(A)I-A\) is the first Newton tensor of \(A\).

\section{Uniform separation} \label{sec:separation}
Throughout this section, \(n\geq3\).  Suppose that
\(u\in C^4(B_{3/2})\) satisfies
\begin{equation}\label{u}
    \sigma_2(\kappa[u])=\lambda_u
    \quad\text{in } B_{3/2},
    \qquad
    \kappa[u]\in\Gamma_2,
\end{equation}
and that
\[
    \lVert Du\rVert_{L^\infty(B_{3/2})}\leq K.
\]
Let \(0<\lambda_v<\lambda_u\), and let
\(v\in C^4(B_1)\cap C^2(\overline B_1)\) be an admissible solution of
\begin{equation}\label{v}
    \begin{cases}
\sigma_2(\kappa[v])=\lambda_v,
            & \text{in } B_1,\\
        \kappa[v]\in\Gamma_2,
            & \text{in } B_1,\\
        v=u,
            & \text{on } \partial B_1.
    \end{cases}
\end{equation}
Since \(\lambda_v<\lambda_u\), the strong comparison principle gives
\[
    v>u
    \quad\text{in } B_1.
\]
Moreover, \(H_v>0\), so \(v\) cannot have an interior maximum.  Hence
\[
 \operatorname{osc}_{B_1}v
 \leq \operatorname{osc}_{B_1}u
 \leq 2K.
\]
Korevaar's interior gradient estimate therefore gives
\[
 \lVert Dv\rVert_{L^\infty(B_{7/8})}
 \leq C(n,K,\lambda_v).
\]
After replacing \(K\) by this larger constant, we assume below that
\(\lVert Dv\rVert_{L^\infty(B_{7/8})}\leq K\).

Write
\[
 X_u(x)=(x,u(x)),\qquad X_v(x)=(x,v(x)),
 \qquad \Sigma_u=X_u(B_1),\qquad \Sigma_v=X_v(B_1),
\]
and denote the downward unit normals by \(N_u,N_v\).  Thus
\[
 N_u=\frac{(Du,-1)}{W_u},\qquad
 N_v=\frac{(Dv,-1)}{W_v}.
\]
On \(\Sigma_v\), let
\[
 T_v^{ij}=H_vg_v^{ij}-h_v^{ij},
 \qquad
 \Delta_{T_v} f=T_v^{ij}\nabla^2_{ij}f.
\]
We use the convention
\[
 h(U,V)=-\langle\overline\nabla_UV,N\rangle,
 \qquad
 \overline\nabla_UN=AU.
\]
Repeated indices are summed, and all covariant derivatives below are
taken on \(\Sigma_v\).

Define the closed hypograph of \(u\) by
\[
    E_u
    =
    \bigl\{(x,z):x\in\overline{B_1},\ z\leq u(x)\bigr\},
\]
and, for \(y\in B_1\), set
\[
    d(y)
    =
    \operatorname{dist}\bigl(X_v(y),E_u\bigr).
\]

We shall prove the following quantitative separation estimate.

\begin{theorem}\label{thm:separation}
Suppose that \(u\) and \(v\) satisfy \eqref{u} and \eqref{v}. Then there is \(C=C(n,K,\lambda_u,\lambda_v)\) such that
\begin{equation}
 v(y)-u(y)\geq d(y)\geq C^{-1}
 \qquad (\forall~ y\in B_{1/2}).
\end{equation}
\end{theorem}

\subsection{The distance function}\label{subsec:The distance function}
We first compare the ambient distance \(d\) quantitatively with the gap \(v-u\) and derive the Lipschitz estimate.
\begin{lemma}\label{lem:distance-gap}
For \(y\in B_1\),
\begin{equation}\label{4.2}
    \frac{v(y)-u(y)}{\sqrt{1+K^2}}
    \leq d(y)
    \leq v(y)-u(y).
\end{equation}
Moreover, \(d\) is \(\sqrt{1+K^2}\)-Lipschitz continuous in
\(B_{7/8}\).
\end{lemma}

\begin{proof}
Since \((y,u(y))\in E_u\), the vertical competitor immediately gives
\[
    d(y)
    \leq
    \bigl|X_v(y)-(y,u(y))\bigr|
    =
    v(y)-u(y).
\]

To prove the opposite inequality, fix \(x\in\overline{B_1}\). If
\(v(y)>u(x)\), then
\begin{align*}
    v(y)-u(y)
    &=
    v(y)-u(x)+u(x)-u(y)\\
    &\leq
    v(y)-u(x)+K|x-y|\\
    &\leq
    \sqrt{1+K^2}\,
    \sqrt{|x-y|^2+\bigl(v(y)-u(x)\bigr)^2}.
\end{align*}
The last square root is precisely the distance from \(X_v(y)\) to the
vertical ray
\[
    \bigl\{(x,z):z\leq u(x)\bigr\}.
\]

If \(v(y)\leq u(x)\), then
\[
    v(y)-u(y)
    \leq
    u(x)-u(y)
    \leq
    K|x-y|
    \leq
    \sqrt{1+K^2}\,|x-y|,
\]
while \(|x-y|\) is the distance from \(X_v(y)\) to the same vertical
ray. Taking the infimum over \(x\in\overline{B_1}\) yields
\[
    v(y)-u(y)
    \leq
    \sqrt{1+K^2}\,d(y),
\]
which proves \eqref{4.2}.

Finally, the distance to a closed set is \(1\)-Lipschitz in the ambient
Euclidean space. Hence, for \(y_1,y_2\in B_{7/8}\),
\begin{align*}
    |d(y_1)-d(y_2)|
    &\leq
    |X_v(y_1)-X_v(y_2)|\\
    &=
    \sqrt{|y_1-y_2|^2
    +|v(y_1)-v(y_2)|^2}\\
    &\leq
    \sqrt{1+K^2}\,|y_1-y_2|,
\end{align*}
where the last inequality follows from
\(\lVert Dv\rVert_{L^\infty(B_{7/8})}\leq K\).
\end{proof}

\begin{remark}\label{rem:intrinsic-one-point-obstruction}
Neither of the inequalities
\begin{equation*}
T_v^{ij}\nabla_{ij}^v(v-u)\leq C,
\qquad
T_u^{ij}\nabla_{ij}^u(v-u)\geq-C
\end{equation*}
is a pointwise consequence of the two curvature equations,
admissibility, and the gradient bounds. Here \(\nabla^u\) and
\(\nabla^v\) denote the Levi--Civita connections of the graphs of
\(u\) and \(v\), respectively.

Indeed, for \(N>1\), consider near the origin in \(\mathbb R^3\)
\begin{align*}
u_N(x)
&=
\sqrt3\,x_1
+\frac12\left[
 \left(\frac{8}{5N}-\frac{32N}{5}\right)x_1^2
 +2Nx_2^2+8Nx_3^2
\right],\\
v_N(x)
&=
1+\frac12\left[
 Nx_1^2
 +\left(\frac{1}{16N}-\frac{3N}{4}\right)x_2^2
 +3Nx_3^2
\right].
\end{align*}
At the origin,
\begin{equation*}
Du_N=\sqrt3\,e_1,
\qquad
Dv_N=0,
\qquad
W_u=2,
\qquad
W_v=1.
\end{equation*}
The corresponding principal curvatures are
\begin{equation*}
\kappa[u_N]
=
\left(
 \frac{1}{5N}-\frac{4N}{5},\,N,\,4N
\right),
\qquad
\kappa[v_N]
=
\left(
 N,\,\frac{1}{16N}-\frac{3N}{4},\,3N
\right).
\end{equation*}
Consequently,
\begin{equation*}
\sigma_2(\kappa[u_N])=1,
\qquad
\sigma_2(\kappa[v_N])=\frac14,
\end{equation*}
while
\begin{equation*}
H[u_N]=\frac{1}{5N}+\frac{21N}{5}>0,
\qquad
H[v_N]=\frac{1}{16N}+\frac{13N}{4}>0.
\end{equation*}
Thus both curvature vectors belong to \(\Gamma_2\).

Set
\begin{equation*}
\rho_N:=v_N-u_N.
\end{equation*}
For a graph \(z\), its Christoffel symbols in graph coordinates are
\begin{equation*}
\Gamma_{ij}^k[z]
=
\frac{z_kz_{ij}}{1+|Dz|^2}.
\end{equation*}
Since \(D\rho_N(0)=-\sqrt3\,e_1\), at the origin we obtain
\begin{equation*}
\bigl(\nabla^{2}_u\rho_N\bigr)_{ij}
=
\left(D^2v_N-\frac14D^2u_N\right)_{ij},
\qquad
\bigl(\nabla^{2}_v\rho_N\bigr)_{ij}
=
\left(D^2v_N-D^2u_N\right)_{ij}.
\end{equation*}
Contracting with the corresponding Newton tensors gives
\begin{align*}
T_u^{ij}\nabla_{ij}^u\rho_N
&=
\frac{-28+N^{-2}-44N^2}{80}
\longrightarrow-\infty,\\
T_v^{ij}\nabla_{ij}^v\rho_N
&=
\frac{-32-N^{-2}+44N^2}{10}
\longrightarrow+\infty.
\end{align*}

For \(n>3\), let \(u_N\) and \(v_N\) be independent of
\(x_4,\ldots,x_n\). The additional principal curvatures are zero, and
all the preceding identities remain unchanged.
\end{remark}

\subsection{The compression inequality}\label{subsec:The compression inequality}
Remark~\ref{rem:intrinsic-one-point-obstruction} shows that a direct
one-point calculation for the vertical gap is not available.  We
therefore retain the two contact points.  The resulting second-variation
calculation leads to a matrix of the form \(M(I-M)^{-1}\), and the
following lemma controls the \((n-1)\)-dimensional block that appears
in this calculation.
\begin{lemma}\label{lem:compression}
Let \(n\geq3\).  Let \(M\in\Gamma_2\) be symmetric and suppose that
\(I-M>0\).  Set
\[
 D=M(I-M)^{-1}.
\]
Then, in any orthonormal basis,
\begin{equation}\label{eq:compression}
 \sum_{2\leq\alpha<\beta\leq n}
 \left(D_{\alpha\alpha}D_{\beta\beta}-D_{\alpha\beta}^2\right)
 +\frac{n-2}{n-1}\sum_{\alpha=2}^nD_{\alpha\alpha}
 \geq \frac{n-2}{n-1}\sigma_2(M).
\end{equation}
\end{lemma}

\begin{proof}
Fix the orthonormal basis \(\{e_1,\ldots,e_n\}\) from the statement and set
\[
 D'=(D_{\alpha\beta})_{2\leq\alpha,\beta\leq n}.
\]
Choose an orthonormal eigenbasis \(\{\xi_j\}\) of \(M\) and write
\[
 M\xi_j=\mu_j\xi_j,
 \qquad \tau_j=\frac{\mu_j}{1-\mu_j},
 \qquad w_j=\langle e_1,\xi_j\rangle^2.
\]
Then \(D\xi_j=\tau_j\xi_j\), \(w_j\geq0\), and \(\sum_jw_j=1\).
The \((1,1)\)-cofactor of \(zI_n-D\) gives the polynomial identity
\[
 \det(zI_{n-1}-D')
 =\langle\operatorname{adj}(zI_n-D)e_1,e_1\rangle
 =\sum_{j=1}^nw_j\prod_{i\neq j}(z-\tau_i).
\]
Comparing coefficients yields
\[
 \sigma_k(D')=\sum_{j=1}^nw_j\sigma_k(\tau|j),
 \qquad k=1,2.
\]
Therefore,
\[
 \sigma_2(D')+\frac{n-2}{n-1}\sigma_1(D')
 =\sum_{j=1}^nw_j
 \left[
  \sigma_2(\tau|j)+\frac{n-2}{n-1}\sigma_1(\tau|j)
 \right].
\]
Since \(w_j\geq0\) and \(\sum_jw_j=1\), it is enough to prove the
following inequality for each \(j\):
\begin{equation}\label{eq:deleted-compression}
 \sigma_2(\tau|j)+\frac{n-2}{n-1}\sigma_1(\tau|j)
 \geq\frac{n-2}{n-1}\sigma_2(M).
\end{equation}

Fix \(j\).  Positivity of the first Newton tensor of \(M\) gives
\(\sum_{i\neq j}\mu_i>0\), while \(I-M>0\) gives \(\mu_i<1\) for every
\(i\).  Keep \(\mu_j\) fixed and regard the remaining eigenvalues as
variables.  Define
\[
 F:=\sigma_2(\tau|j)+\frac{n-2}{n-1}\sigma_1(\tau|j)
 -\frac{n-2}{n-1}\sigma_2(M).
\]
Then \eqref{eq:deleted-compression} is exactly the assertion \(F\geq0\).
For distinct \(r,s\neq j\), direct differentiation gives
\[
 \frac{\partial_rF-\partial_sF}{\mu_r-\mu_s}
 =\frac{n-2}{n-1}
 +\frac{1+(2-\mu_r-\mu_s)
 \left(\displaystyle\sum_{k\neq j,r,s}\frac1{1-\mu_k}
       -\frac{(n-2)^2}{n-1}\right)}
 {(1-\mu_r)^2(1-\mu_s)^2},
\]
where the quotient is understood by continuity if \(\mu_r=\mu_s\).
Since $\sum_{i\neq j}\mu_i>0$ and $\mu_k<1$, we have
\[
 n-3+\mu_r+\mu_s
 >
 \sum_{k\neq j,r,s}(1-\mu_k)
 \geq0.
\]
For $n>3$, the Cauchy--Schwarz inequality therefore gives
\[
 \sum_{k\neq j,r,s}\frac1{1-\mu_k}
 \geq
 \frac{(n-3)^2}{n-3+\mu_r+\mu_s}.
\]
For $n=3$, the sum is empty and both sides of this inequality
are zero, so the same bound holds.  Consequently,
\begin{align*}
 &1+(2-\mu_r-\mu_s)
 \left(\sum_{k\neq j,r,s}\frac1{1-\mu_k}
       -\frac{(n-2)^2}{n-1}\right)\\
 &\qquad\geq
 \frac{[(n-1)-(n-2)(2-\mu_r-\mu_s)]^2}
 {(n-1)(n-3+\mu_r+\mu_s)}\geq0.
\end{align*}
Thus \((\mu_r-\mu_s)(\partial_rF-\partial_sF)\geq0\). Consequently, replacing any two entries by their average while leaving
the other entries unchanged cannot increase \(F\).
Repeatedly applying this operation to a largest and a smallest entry
produces a sequence converging to the constant vector, and therefore
\[
 F\geq F(q,\ldots,q),
 \qquad q=\frac1{n-1}\sum_{i\neq j}\mu_i\in(0,1).
\]
A direct calculation now yields
\[
 F(q,\ldots,q)
 =\frac{(n-2)q^2}{2(1-q)^2}
 \bigl[(n-2)q(2-q)+3-2q\bigr]
 +(n-2)(1-\mu_j)q\geq0.
\]
This proves \eqref{eq:deleted-compression}; summing it with weights \(w_j\)
proves \eqref{eq:compression}.
\end{proof}

\subsection{The two-surface inequality}\label{subsec:The two-surface inequality}
We now return to the distance \(d\).  Since \(d\) need not be smooth,
we express its differential inequality using a smooth test function at
an interior closest pair.  Lemma~\ref{lem:compression} provides the estimate needed in the second-variation calculation.
\begin{proposition}\label{prop:distance-jacobi}
Let \(X=X_u(x)\in\Sigma_u\) and \(Y=X_v(y)\in\Sigma_v\) be interior
points such that
\[
    |Y-X|=\operatorname{dist}(Y,E_u)=d>0,
\]
and suppose that \(\phi\) is \(C^2\) in a neighborhood of \(Y\) and
\[
    (X',Y')
    \longmapsto
    \phi(Y')+\log|Y'-X'|
\]
has a local minimum at \((X,Y)\). Then
\begin{equation}\label{eq:strong-two-surface-jacobi}
    \Delta_{T_v}\phi
    \geq
    \frac{1}{2(n-1)}
    T_v^{ij}\phi_i\phi_j
    +
    \frac{\sqrt{\lambda_u\lambda_v}-\lambda_v}{d}.
\end{equation}
\end{proposition}

\begin{proof}
Choose local orthonormal tangent frames $\{X_a\}$ on $\Sigma_u$
and $\{Y_i\}$ on $\Sigma_v$ such that
\[
 \nabla^{\Sigma_u}_{X_b}X_a=0\quad\text{at }X,
 \qquad
 \nabla^{\Sigma_v}_{Y_j}Y_i=0\quad\text{at }Y.
\]
Their values at $X$ and $Y$ will be chosen below.
All quantities in this proof are evaluated at these points.
Thus
\[
 \langle X_a,X_b\rangle=\delta_{ab},\qquad
 \langle Y_i,Y_j\rangle=\delta_{ij},
\]
and
\[
 X_{ab}:=\overline\nabla_{X_b}X_a=-h^u_{ab}N_u,
 \qquad
 Y_{ij}:=\overline\nabla_{Y_j}Y_i=-h^v_{ij}N_v.
\]

Put
\[
 \Psi(X',Y')=\phi(Y')+\log|Y'-X'|.
\]
Subscripts on $\Psi$ denote covariant derivatives on
$\Sigma_u\times\Sigma_v$ with respect to these frames.
Differentiating in the directions $X_a$ gives
\[
 0=\Psi_a
 =-\frac{\langle Y-X,X_a\rangle}{d^2}.
\]
Thus $Y-X$ is normal to $\Sigma_u$. Since $X$ is a nearest point
of $Y$ on the boundary of $E_u$, the downward orientation gives
\[
 Y-X=-dN_u.
\]
In the rest of the proof, denote
\[
 c=\langle N_u,N_v\rangle,\qquad
 p_{ai}=\langle X_a,Y_i\rangle,\qquad
 \nu_i=\langle N_u,Y_i\rangle.
\]
Differentiating in the directions $Y_i$ gives
\[
 0=\Psi_i
 =\phi_i+\frac{\langle Y-X,Y_i\rangle}{d^2},
\]
and hence
\begin{equation}\label{eq:first-local}
 \phi_i=\frac{\nu_i}{d}.
\end{equation}

Taking covariant derivatives once more, we obtain
\[
 \Psi_{ab}=\frac{\delta_{ab}-dh^u_{ab}}{d^2},
 \qquad
 \Psi_{aj}=-\frac{p_{aj}}{d^2},
\]
and
\[
 \Psi_{ij}
 =
 \nabla_i\nabla_j\phi
 +\frac{c}{d}h^v_{ij}
 +\frac{\delta_{ij}-2\nu_i\nu_j}{d^2}.
\]
Since $\Psi$ has a local minimum at $(X,Y)$, its Hessian is
nonnegative. Therefore
\begin{equation}\label{eq:block-local}
 \begin{pmatrix}
 \displaystyle
 d^{-2}(\delta_{ab}-dh^u_{ab})
 &
 \displaystyle
 -d^{-2}p_{aj}
 \\[2mm]
 \displaystyle
 -d^{-2}p_{bi}
 &
 \displaystyle
 \nabla_i\nabla_j\phi
 +d^{-1}ch^v_{ij}
 +d^{-2}(\delta_{ij}-2\nu_i\nu_j)
 \end{pmatrix}
 \geq0.
\end{equation}

Since $N_u$ and $N_v$ are both downward graph normals, $c>-1$.
For $Z\in T_X\Sigma_u$, define
\[
 \mathcal RZ
 =
 Z-\frac{\langle N_v,Z\rangle}{1+c}(N_u+N_v).
\]
For $Z,W\in T_X\Sigma_u$, a direct calculation gives
\[
 \langle\mathcal RZ,N_v\rangle=0,
 \qquad
 \langle\mathcal RZ,\mathcal RW\rangle=\langle Z,W\rangle.
\]
Thus $\mathcal R:T_X\Sigma_u\to T_Y\Sigma_v$ is the canonical
isometry induced by the shortest ambient rotation carrying $N_u$
to $N_v$. If $|c|<1$, choose the frame values
\[
 X_1=\frac{N_v-cN_u}{\sqrt{1-c^2}},
 \qquad
 Y_1:=\mathcal RX_1
 =\frac{cN_v-N_u}{\sqrt{1-c^2}},
\]
and $X_\alpha=Y_\alpha$, $2\leq\alpha\leq n$, in the common
tangent subspace. If $c=1$, the two tangent spaces agree, and we
take the same orthonormal basis in both spaces. In either case,
\begin{equation}\label{eq:adapted-local}
\begin{gathered}
 p_{11}=c,\qquad
 p_{\alpha\beta}=\delta_{\alpha\beta},\qquad
 p_{1\alpha}=p_{\alpha1}=0,\\
 \nu_1=-\sqrt{1-c^2},\qquad
 \nu_\alpha=0
 \qquad(2\leq\alpha\leq n).
\end{gathered}
\end{equation}

Set
\[
 M_{ij}=dh^u_{ij}.
\]
The upper-left block of \eqref{eq:block-local} gives $I-M\geq0$,
while
\[
 M\in\Gamma_2,\qquad
 \sigma_2(M)=d^2\lambda_u.
\]
For $\varepsilon>0$, set
\[
 D_\varepsilon=M\bigl((1+\varepsilon)I-M\bigr)^{-1},
 \qquad
 A_\varepsilon=\frac{M}{1+\varepsilon}.
\]
Then
\[
 D_\varepsilon,A_\varepsilon\in\Gamma_2,\qquad
 \sigma_2(D_\varepsilon)
 \geq\frac{\sigma_2(M)}{(1+\varepsilon)^2}
 =\frac{d^2\lambda_u}{(1+\varepsilon)^2},
\]
and
\[
 P_\varepsilon:=D_\varepsilon-A_\varepsilon
 =
 \frac1{1+\varepsilon}M^2
 \bigl((1+\varepsilon)I-M\bigr)^{-1}\geq0.
\]

Replacing the upper-left block of \eqref{eq:block-local} by
$d^{-2}((1+\varepsilon)I-M)$ preserves nonnegativity.
Since $(1+\varepsilon)I-M>0$, we may minimize the resulting
quadratic form over its $T_X\Sigma_u$ component. Using
\[
 \bigl((1+\varepsilon)I-M\bigr)^{-1}
 =
 \frac1{1+\varepsilon}(I+D_\varepsilon),
\]
we obtain
\begin{equation}\label{eq:regularized-schur}
\begin{aligned}
 \nabla_i\nabla_j\phi
 \geq{}&
 -\frac{c}{d}h^v_{ij}
 +\frac{\nu_i\nu_j}{d^2}
 -\frac{\varepsilon}{(1+\varepsilon)d^2}\delta_{ij}\\
 &+\frac1{(1+\varepsilon)d^2}
 p_{ai}(D_\varepsilon)_{ab}p_{bj}.
\end{aligned}
\end{equation}

Define $B$ by
\[
\begin{aligned}
 dB_{11}
 &=
 \frac1{1+\varepsilon}
 \left[
 c^2(D_\varepsilon)_{11}
 +\frac{n-2}{n-1}(1-c^2)
 \right],\\
 dB_{1\alpha}
 &=\frac{c}{1+\varepsilon}(D_\varepsilon)_{1\alpha},\\
 dB_{\alpha\beta}
 &=\frac1{1+\varepsilon}(D_\varepsilon)_{\alpha\beta},
 \qquad 2\leq\alpha,\beta\leq n.
\end{aligned}
\]
Expanding $\sigma_2(B)$ and applying Lemma~\ref{lem:compression}
to $M/(1+\varepsilon)$, whose associated matrix is $D_\varepsilon$,
gives
\begin{align*}
 d^2(1+\varepsilon)^2\sigma_2(B)
 ={}&
 c^2\sigma_2(D_\varepsilon)\\
 &+(1-c^2)
 \left[
 \sum_{2\leq\alpha<\beta\leq n}
 \left(
 (D_\varepsilon)_{\alpha\alpha}
 (D_\varepsilon)_{\beta\beta}
 -(D_\varepsilon)_{\alpha\beta}^2
 \right)
 +\frac{n-2}{n-1}
 \sum_{\alpha=2}^n(D_\varepsilon)_{\alpha\alpha}
 \right]\\
 \geq{}&
 \left[
 c^2+\frac{n-2}{n-1}(1-c^2)
 \right]
 \frac{d^2\lambda_u}{(1+\varepsilon)^2}.
\end{align*}
Consequently,
\[
 \sigma_2(B)
 \geq
 \frac{n-2+c^2}{(n-1)(1+\varepsilon)^4}\lambda_u.
\]
Moreover,
\[
\begin{aligned}
 d\operatorname{tr}B
 =
 \frac1{1+\varepsilon}
 \left\{
 c^2\operatorname{tr}D_\varepsilon
 +(1-c^2)
 \left[
 \frac{n-2}{n-1}
 +\sum_{\alpha=2}^n(D_\varepsilon)_{\alpha\alpha}
 \right]
 \right\}.
\end{aligned}
\]
Since $D_\varepsilon\in\Gamma_2$,
\[
 \operatorname{tr}D_\varepsilon>0,\qquad
 \sum_{\alpha=2}^n(D_\varepsilon)_{\alpha\alpha}
 =T_1(D_\varepsilon)(e_1,e_1)>0.
\]
It follows that $\operatorname{tr}B>0$. Together with
$\sigma_2(B)>0$, this proves $B\in\Gamma_2$.

By \eqref{eq:first-local} and \eqref{eq:adapted-local},
\[
 T_v^{ij}\phi_i\phi_j
 =
 \frac{1-c^2}{d^2}T_v^{11}.
\]
Contracting \eqref{eq:regularized-schur} with $T_v$ and using the
definition of $B$ yields
\begin{align*}
 \Delta_{T_v}\phi
 \geq{}&
 \left(
 1-\frac{n-2}{(n-1)(1+\varepsilon)}
 \right)T_v^{ij}\phi_i\phi_j
 +\frac1dT_v^{ij}B_{ij}
 -\frac{2c\lambda_v}{d}\\
 &-\frac{\varepsilon}{(1+\varepsilon)d^2}
 \operatorname{tr}T_v.
\end{align*}
The coefficient of the gradient term is at least $1/(n-1)$.
The G{\aa}rding inequality implies
\[
 T_v^{ij}B_{ij}
 \geq
 \frac{2}{(1+\varepsilon)^2}
 \sqrt{\lambda_u\lambda_v}
 \sqrt{\frac{n-2+c^2}{n-1}}.
\]
Letting $\varepsilon\downarrow0$, we obtain
\begin{equation}\label{eq:first-estimate-local}
 \Delta_{T_v}\phi
 \geq
 \frac1{n-1}T_v^{ij}\phi_i\phi_j
 +\frac2d
 \left[
 \sqrt{\lambda_u\lambda_v}
 \sqrt{\frac{n-2+c^2}{n-1}}
 -c\lambda_v
 \right].
\end{equation}

It remains to estimate the zero-order term. Since
\[
 \sqrt{\frac{n-2+c^2}{n-1}}
 \geq
 \sqrt{\frac{n-2}{n-1}}>\frac12
\]
and
\[
 \frac{n-2+c^2}{n-1}-c^2
 =
 \frac{(n-2)(1-c^2)}{n-1}\geq0,
\]
we have
\[
 \sqrt{\frac{n-2+c^2}{n-1}}\geq |c|\geq c.
\]
Using $\lambda_u>\lambda_v$, we obtain
\begin{align*}
 &\frac2d
 \left[
 \sqrt{\lambda_u\lambda_v}
 \sqrt{\frac{n-2+c^2}{n-1}}
 -c\lambda_v
 \right]
 -\frac{\sqrt{\lambda_u\lambda_v}-\lambda_v}{d}\\
 &\quad=
 \frac{\lambda_v}{d}
 \left[
 \sqrt{\frac{\lambda_u}{\lambda_v}}
 \left(
 2\sqrt{\frac{n-2+c^2}{n-1}}-1
 \right)
 +1-2c
 \right]\\
 &\quad\geq
 \frac{\lambda_v}{d}
 \left[
 2\sqrt{\frac{n-2+c^2}{n-1}}-1+1-2c
 \right]\\
 &\quad=
 \frac{2\lambda_v}{d}
 \left[
 \sqrt{\frac{n-2+c^2}{n-1}}-c
 \right]\geq0.
\end{align*}
Consequently, \eqref{eq:first-estimate-local} gives the stronger
inequality
\[
 \Delta_{T_v}\phi
 \geq
 \frac1{n-1}T_v^{ij}\phi_i\phi_j
 +\frac{\sqrt{\lambda_u\lambda_v}-\lambda_v}{d}.
\]
Since $T_v$ is positive definite, this implies
\eqref{eq:strong-two-surface-jacobi}.

If $Y$ has several nearest points on the lower hypersurface,
the same argument applies to each contact pair separately.
\end{proof}

\begin{remark}[Role of the compression inequality]
The matrix
$\bigl(p_{ai}(D_\varepsilon)_{ab}p_{bj}\bigr)$
in \eqref{eq:regularized-schur} need not belong to $\Gamma_2$.
For example, when $n=3$, take
\[
 M=\operatorname{diag}\left(\frac12,\frac14,-\frac1{10}\right).
\]
Then $M\in\Gamma_2$, $I-M>0$, and
\[
 \sigma_2(M)=\frac1{20},\qquad
 D_\varepsilon
 =\operatorname{diag}\left(
 \frac1{1+2\varepsilon},
 \frac1{3+4\varepsilon},
 -\frac1{11+10\varepsilon}
 \right).
\]
Although $D_\varepsilon\in\Gamma_2$, when $c=0$,
\eqref{eq:adapted-local} gives
\[
 \sigma_2\!\left(
 \bigl(p_{ai}(D_\varepsilon)_{ab}p_{bj}\bigr)
 \right)
 =-\frac1{(3+4\varepsilon)(11+10\varepsilon)}<0.
\]
Thus identifying the tangent spaces by $\mathcal R$ does not
by itself justify applying the mixed G{\aa}rding inequality
to this matrix.

The term $\nu_i\nu_j$ in \eqref{eq:regularized-schur} is supported
in the first direction, where the projection causes this loss.
The definition of $B$ can be written as
\[
 dB_{ij}
 =\frac1{1+\varepsilon}
 \left[
 p_{ai}(D_\varepsilon)_{ab}p_{bj}
 +\frac{n-2}{n-1}\nu_i\nu_j
 \right].
\]
Lemma~\ref{lem:compression} supplies the lower bound
\[
 \sigma_2(B)
 \ge
 \frac{n-2+c^2}{(n-1)(1+\varepsilon)^4}\lambda_u.
\]
Together with $\operatorname{tr}B>0$, this gives $B\in\Gamma_2$
and permits the mixed G{\aa}rding comparison.
At the same time, the coefficient of
$T_v^{ij}\phi_i\phi_j$ left in the estimate is
\[
 1-\frac{n-2}{(n-1)(1+\varepsilon)}
 \ge\frac1{n-1}.
\]
The compression inequality therefore compensates for the
mismatch of the normals while preserving a positive gradient
term in the Jacobi inequality.
\end{remark}

\subsection{A positive lower bound at some point}\label{Subsec:A positive lower bound at some point}
To establish the separation estimate throughout \(B_{1/2}\), we first
need a quantitative starting point: \(d\) must be bounded below
somewhere in a fixed interior ball. The next lemma obtains this from
the two-surface inequality.
\begin{lemma}\label{lem:seed}
There is \(C=C(n,K,\lambda_u,\lambda_v)\) such that
\[
 \max_{\overline{B}_{1/8}}d\geq C^{-1}.
\]
\end{lemma}

\begin{proof}
On \(X_v(B_{1/8})\), solve
\[
 \begin{cases}
 -\Delta_{T_v}q=1&\text{in }X_v(B_{1/8}),\\
 q=0&\text{on }X_v(\partial B_{1/8}).
\end{cases}
\]
For the fixed smooth function \(v\), the tensor \(T_v\) is positive
definite and hence uniformly elliptic on this compact subdomain.
Thus the problem has a classical solution, and \(q>0\) in the interior.
We first obtain a lower bound for \(\max q\) without using the
ellipticity ratio of \(T_v\).

The Newton identities are
\[
 \nabla_iT_v^{ij}=0,\qquad
 \operatorname{tr}T_v=(n-1)H_v,\qquad
 T_v^{ij}(h_v)_{ij}=2\lambda_v.
\]
Choose \(0\leq\eta\in C_c^\infty(B_{1/8})\), with
\(\eta=1\) on \(B_{1/16}\), and lift it to the graph.  Integration by
parts gives
\[
 \int_{X_v(B_{1/8})}\eta\,d\mu
 =
 -\int_{X_v(B_{1/8})}q\,\Delta_{T_v}\eta\,d\mu,
\]
and consequently
\[
 |B_{1/16}|
 \leq \int_{X_v(B_{1/8})}\eta\,d\mu
 \leq (\max q)
       \int_{X_v(B_{1/8})}|\Delta_{T_v}\eta|\,d\mu.
\]
Moreover,
\[
 |\Delta_{T_v}\eta|
 \leq (n-1)H_v|D^2\eta|+2\lambda_v|D\eta|.
\]
Since \(d\mu=W_v\,dx\), \(W_v\leq\sqrt{1+K^2}\), and
\(H_v=\operatorname{div}(Dv/W_v)>0\),
\[
 \begin{aligned}
 \int_{X_v(B_{1/8})}H_v\,d\mu
 &\leq
 \lVert W_v\rVert_{L^\infty}
 \int_{B_{1/8}}H_v\,dx\\
 &=
 \lVert W_v\rVert_{L^\infty}
 \int_{\partial B_{1/8}}
 \frac{Dv}{W_v}\cdot\nu_{\partial B_{1/8}}\,dS\\
 &\leq
 \sqrt{1+K^2}\,|\partial B_{1/8}|.
 \end{aligned}
\]
Also,
\[
 \operatorname{area}\bigl(X_v(B_{1/8})\bigr)
 \leq
 \sqrt{1+K^2}\,|B_{1/8}|.
\]
Consequently,
\begin{equation}\label{eq:maxq}
 \max q\geq C(n,K,\lambda_v)^{-1}.
\end{equation}

Set
\[
 a=\frac1{4(n-1)},\qquad
 \theta=
 \frac a2\bigl(\sqrt{\lambda_u\lambda_v}-\lambda_v\bigr)
 \left(\max_{\overline B_{1/8}}d\right)^{a-1}.
\]
If a nearest point to some \(X_v(y)\), \(y\in B_{1/8}\), lies on the
lateral boundary of \(E_u\), then \(d(y)\geq7/8\), and the conclusion
follows.  We may therefore assume
\[
 \max_{\overline B_{1/8}}d<\frac78.
\]
Every nearest point used below then lies in the interior of
\(\Sigma_u\).

We \textbf{claim} that
\begin{equation}\label{eq:comparison-dq}
 d^a\geq\theta q
 \qquad\text{in }B_{1/8}.
\end{equation}
Otherwise \(d^a-\theta q\) has a negative interior minimum at \(y_*\).
Indeed, \(q=0\) on the boundary, so a negative minimum cannot occur
there.  Its negativity also implies
\[
 d(y_*)^a-\theta q(y_*)<0.
\]
The function
\[
 \psi(y)=\theta q(y)+d(y_*)^a-\theta q(y_*)
\]
then satisfies
\[
 \psi\leq d^a,\qquad \psi(y_*)=d(y_*)^a>0.
\]
Thus \(\psi>0\) near \(y_*\), which is precisely where the logarithm
below is used.
Choose a nearest point \(X_u(x_*)\) to \(X_v(y_*)\). It follows that
\[
 (x,y)\longmapsto
 -\frac1a\log\psi(y)+\log|X_v(y)-X_u(x)|
\]
has a local minimum at \((x_*,y_*)\). Applying Proposition~\ref{prop:distance-jacobi} at
\((x_*,y_*)\) with \(\phi=-a^{-1}\log\psi\) gives
\begin{align*}
 \Delta_{T_v}\psi
 &=
 -a\psi\,\Delta_{T_v}\phi
 +a^2\psi T_v^{ij}\phi_i\phi_j\\
 &\leq
 -a\bigl(\sqrt{\lambda_u\lambda_v}-\lambda_v\bigr)
 \frac{\psi}{d}
 +a\left(a-\frac1{2(n-1)}\right)
 \psi T_v^{ij}\phi_i\phi_j\\
 &\leq
 -a\bigl(\sqrt{\lambda_u\lambda_v}-\lambda_v\bigr)d^{a-1}\\
 &\leq
 -a\bigl(\sqrt{\lambda_u\lambda_v}-\lambda_v\bigr)
 \left(\max_{\overline B_{1/8}}d\right)^{a-1}
 =-2\theta.
\end{align*}
On the other hand, \(\Delta_{T_v}\psi=-\theta\), a contradiction.  This
proves \eqref{eq:comparison-dq}.

Evaluating \eqref{eq:comparison-dq} at a maximum point of \(q\) and recalling
the definition of \(\theta\), we obtain
\[
 \left(\max_{\overline B_{1/8}}d\right)^a
 \geq
 \frac a2\bigl(\sqrt{\lambda_u\lambda_v}-\lambda_v\bigr)
 \left(\max_{\overline B_{1/8}}d\right)^{a-1}\max q.
\]
After enlarging \(C\), this gives
\[
 \max_{\overline B_{1/8}}d
 \geq C(n,K,\lambda_u,\lambda_v)^{-1}.
\]
\end{proof}

Set
\[
 \mathcal G(p)=\sqrt{1+|p|^2}-1,\qquad
 \mathcal G(Dv)=\frac1{\langle N_v,E_{n+1}\rangle}-1.
\]
Fix $y_0\in\overline B_{1/8}$. For $y\in B_{2/3}(y_0)$, set
\[
 z=y-y_0,\qquad R_i=\langle(z,0),e_i\rangle,\qquad
 Z_i=\frac{1}{\langle N_v,E_{n+1}\rangle}\left[
 R_i+\frac{\langle e_i,E_{n+1}\rangle}{\langle N_v,E_{n+1}\rangle}
       \bigl(|z|-\langle(z,0),N_v\rangle\bigr)\right].
\]
We now introduce the following important auxiliary function
$$
\Phi=-\frac{\langle X_v(y)-X_v(y_0),N_v\rangle}
                 {\langle N_v,E_{n+1}\rangle}+|z|\mathcal G(Dv).
$$
Then $\Phi(y_0)=0$, and $|\Phi|\le C(K)|z|$. By direct differentiation, for \(y\neq y_0\), we have
\begin{equation}\label{eq:phi-gradient}
 \Phi_i=\frac{\mathcal G(Dv)}{|z|}R_i-\sum_k h_{ik}Z_k.
\end{equation}
The estimates needed below are collected in the following lemma.
\begin{lemma}\label{lem:phi}
In the principal frame, where $h_v(e_i,e_j)=\kappa_i\delta_{ij}$,
\begin{equation}\label{eq:phi-bound}
 \Delta_{T_v}\Phi
 \ge\frac{|z|}{8(1+K^2)}\sum_i(H_v-\kappa_i)\kappa_i^2
 -\frac{C(n,K,\lambda_v)H_v}{|z|}.
\end{equation}
Moreover,
\begin{equation}\label{eq:R-lower-bound}
 \frac{|z|^2}{1+K^2}\le\sum_iR_i^2\le|z|^2,
\end{equation}
and
\begin{equation}\label{eq:Z-identities}
 \sum_iR_iZ_i\ge\frac{|z|^2}{\sqrt{1+K^2}+K},\qquad
 \sum_iZ_i^2\le4(1+K^2)^2|z|^2.
\end{equation}    
\end{lemma}
\begin{proof}
Since the vector $e_k-\frac{\langle e_k,E_{n+1}\rangle}{\langle N_v,E_{n+1}\rangle}N_v$
is orthogonal to $E_{n+1}$, we have
\[
\begin{aligned}
 &\langle X_v(y)-X_v(y_0),e_k\rangle
 -\frac{\langle e_k,E_{n+1}\rangle}{\langle N_v,E_{n+1}\rangle}
       \langle X_v(y)-X_v(y_0),N_v\rangle\\
 &\qquad=R_k-\frac{\langle e_k,E_{n+1}\rangle}
                        {\langle N_v,E_{n+1}\rangle}
                   \langle(z,0),N_v\rangle.
\end{aligned}
\]
Direct computation yields
\[
\begin{aligned}
 &\langle N_v,E_{n+1}\rangle_j
 =\sum_k h_{jk}\langle e_k,E_{n+1}\rangle,\\
 &\langle e_k,E_{n+1}\rangle_j
 =-h_{kj}\langle N_v,E_{n+1}\rangle,\quad  (|z|)_j=R_j/|z|,\\
 &(R_i)_j=\delta_{ij}
 -\langle e_i,E_{n+1}\rangle\langle e_j,E_{n+1}\rangle
 -h_{ij}\langle(z,0),N_v\rangle,\\
 &\langle(z,0),N_v\rangle_j
 =-\langle e_j,E_{n+1}\rangle\langle N_v,E_{n+1}\rangle
 +\sum_k h_{jk}R_k.
\end{aligned}
\]
Choose the principal frame at this point. The equation and
Codazzi equations give
\[
 T_v^{ij}h_{ij}=2\lambda_v,\qquad
 T_v^{ij}h_{ikj}=T_v^{ij}h_{ijk}=0.
\]
Differentiating \eqref{eq:phi-gradient} and contracting with $T_v^{ij}$, we obtain
\begin{equation}\label{eq:phi-contraction}
\begin{aligned}
 \Delta_{T_v}\Phi={}&
 \sum_i(H_v-\kappa_i)\kappa_i^2
 \Biggl\{\frac{|z|}{\langle N_v,E_{n+1}\rangle}
 \left(1+\frac{2\langle e_i,E_{n+1}\rangle^2}
                   {\langle N_v,E_{n+1}\rangle^2}\right)\\
 &\hspace{9mm}
 +\frac{2\langle e_i,E_{n+1}\rangle}{\langle N_v,E_{n+1}\rangle^2}
 \left[R_i-\frac{\langle e_i,E_{n+1}\rangle}{\langle N_v,E_{n+1}\rangle}
                    \langle(z,0),N_v\rangle\right]\Biggr\}\\
 &-\frac{2}{|z|\langle N_v,E_{n+1}\rangle^2}
       \sum_i(H_v-\kappa_i)\kappa_i\langle e_i,E_{n+1}\rangle R_i-\frac{2\lambda_v}{\langle N_v,E_{n+1}\rangle}\\
 &+\frac{\mathcal G(Dv)}{|z|}
 \left[\sum_i(H_v-\kappa_i)
 \left(1-\langle e_i,E_{n+1}\rangle^2-\frac{R_i^2}{|z|^2}\right)
 -2\lambda_v\langle(z,0),N_v\rangle\right].
\end{aligned}
\end{equation}
Since
\[
 \left|e_i-\frac{\langle e_i,E_{n+1}\rangle}
                     {\langle N_v,E_{n+1}\rangle}N_v\right|^2
 =1+\frac{\langle e_i,E_{n+1}\rangle^2}{\langle N_v,E_{n+1}\rangle^2},
 \qquad
 \frac{|\langle e_i,E_{n+1}\rangle|}{\langle N_v,E_{n+1}\rangle}
 \le |Dv|\le K,
\]
the expression in braces in \eqref{eq:phi-contraction} is at least
\[
\frac{|z|}{\langle N_v,E_{n+1}\rangle}
\left(\sqrt{1+\frac{\langle e_i,E_{n+1}\rangle^2}
{\langle N_v,E_{n+1}\rangle^2}}
-\frac{|\langle e_i,E_{n+1}\rangle|}
{\langle N_v,E_{n+1}\rangle}\right)^2
 \ge\frac{|z|}{4(1+K^2)}.
\]
Using $|R_i|\le|z|$ and
$\langle N_v,E_{n+1}\rangle^{-1}\le\sqrt{1+K^2}$,
Young's inequality gives
\[
\begin{aligned}
 -&\frac{2}{|z|\langle N_v,E_{n+1}\rangle^2}
       \sum_i(H_v-\kappa_i)\kappa_i\langle e_i,E_{n+1}\rangle R_i\\
 &\ge-2K\sqrt{1+K^2}\sum_i(H_v-\kappa_i)|\kappa_i|\\
 &\ge-\frac{|z|}{8(1+K^2)}\sum_i(H_v-\kappa_i)\kappa_i^2
       -\frac{C(n,K)H_v}{|z|}.
\end{aligned}
\]
Since $E_{n+1}\perp(z,0)$,
\[
 1-\langle e_i,E_{n+1}\rangle^2-\frac{R_i^2}{|z|^2}\ge0.
\]
Together with $\mathcal G(Dv)\ge0$, this bounds the remaining terms in
\eqref{eq:phi-contraction} from below by
\[
\begin{aligned}
 &-\frac{2\lambda_v}{\langle N_v,E_{n+1}\rangle}-\frac{2\lambda_v\mathcal G(Dv)}{|z|}\langle(z,0),N_v\rangle
 \\
 &\qquad\ge-2\lambda_v
       \left(\frac{2}{\langle N_v,E_{n+1}\rangle}-1\right)
 \ge-\frac{C(K,\lambda_v)H_v}{|z|}.
\end{aligned}
\]
Here $H_v^2=\sum_i\kappa_i^2+2\lambda_v\ge2\lambda_v$ and
$0<|z|<2/3$. Combining the estimates proves \eqref{eq:phi-bound}.

Finally, $\sum_i e_i\otimes e_i=I-N_v\otimes N_v$ gives
\[
\begin{aligned}
 &\sum_iR_i^2
 =|z|^2-\langle(z,0),N_v\rangle^2
 =|z|^2-\frac{(z\cdot Dv)^2}{1+|Dv|^2},\\
 &\sum_iR_i\langle e_i,E_{n+1}\rangle
 =-\langle(z,0),N_v\rangle\langle N_v,E_{n+1}\rangle,\\
 &\sum_i\langle e_i,E_{n+1}\rangle^2
 =1-\langle N_v,E_{n+1}\rangle^2.
\end{aligned}
\]
Since $|z\cdot Dv|\le|z||Dv|$ and $|Dv|\le K$, the first identity yields \eqref{eq:R-lower-bound}. Substituting these identities into the definition of $Z_i$, we obtain
\[
\begin{aligned}
 \sum_iR_iZ_i
 &=\frac{|z|}{\langle N_v,E_{n+1}\rangle}
       \bigl(|z|-\langle(z,0),N_v\rangle\bigr)\\
 &=\frac{|z|^2}{\langle N_v,E_{n+1}\rangle}
             -|z|\,z\cdot Dv,\\
 \sum_iZ_i^2
 &=\frac{1}{\langle N_v,E_{n+1}\rangle^4}
       \bigl(|z|-\langle(z,0),N_v\rangle\bigr)^2.
\end{aligned}
\]
Now \eqref{eq:Z-identities} follows from $|\langle(z,0),N_v\rangle|\le|z|$,
$\langle N_v,E_{n+1}\rangle^{-1}\le\sqrt{1+K^2}$, and
\[
 \frac1{\langle N_v,E_{n+1}\rangle}-|Dv|
 =\frac1{\sqrt{1+|Dv|^2}+|Dv|}
 \ge\frac1{\sqrt{1+K^2}+K}.
\]
\end{proof}

We now combine Lemma \ref{lem:seed}, Lemma \ref{lem:phi}, the Lipschitz bound for \(d\), and
the two-surface inequality in a maximum-principle argument
to obtain a uniform lower bound for $d$ in \(B_{1/2}\).
\begin{proof}[Proof of Theorem \ref{thm:separation}]
By Lemma \ref{lem:seed}, there are
$M_0=M_0(n,K,\lambda_u,\lambda_v)\in(0,1/16]$ and
$y_0\in\overline B_{1/8}$ such that $d(y_0)\ge M_0$.
Set
\begin{equation}\label{eq:seed-ball}
 r=\min\left\{\frac{M_0}{2\sqrt{1+K^2}},\frac1{40}\right\}.
\end{equation}
Then $d\geq\frac{M_0}{2}$ in $B_r(y_0)$. Use the notation of Lemma~\ref{lem:phi}, with $\beta$ chosen below,
on $B_{2/3}(y_0)\subset B_{19/24}\subset B_{7/8}$, and put
\[
 \rho=\frac49-|z|^2,\qquad b=-\log d.
\]
Initially choose
\begin{equation}\label{eq:preliminary-parameters}
 0<\beta\le1,\qquad
 \beta\bigl(\sqrt{1+K^2}-1\bigr)\le r,\qquad
 S\ge\max\left\{4(n-1),\log\frac2{M_0}\right\}.
\end{equation}
Their final values will be fixed later. Let \(\bar y\) be a maximum point of
\begin{equation}\label{eq:aux}
 2\log\rho+\beta\Phi+\log\max\{b,S\}\quad\hbox{in }B_{2/3}(y_0).
\end{equation}
Since \(\rho\) vanishes on the boundary, the
maximum is attained in the interior. If
\(b(\bar y)\leq S\), then for \(y\in B_{1/2}\),
\[
 \rho(y)^2e^{\beta\Phi(y)}\max\{b(y),S\}
 \leq\rho(\bar y)^2e^{\beta\Phi(\bar y)}S\leq C(K)S,
\]
and hence \(b\leq C(K)S\) in \(B_{1/2}\).

We may therefore assume that \(b(\bar y)>S\). Moreover, \eqref{eq:seed-ball} yields
\[
 |\bar y-y_0|\geq r.
\]
At \(\bar y\), this implies \(|z|\geq r>0\), so all the preceding
differential identities are valid in a neighborhood of \(\bar y\).
For \(y\in B_{2/3}(y_0)\), the horizontal distance from \(y\) to
\(\partial B_1\) is at least \(5/24\). At \(\bar y\),
\[
 d(\bar y)=e^{-b(\bar y)}<e^{-S}
 \leq\frac{M_0}{2}\leq\frac1{32}<\frac5{24}.
\]
Hence every nearest point to \(X_v(\bar y)\) lies in the interior of
\(\Sigma_u\).

The maximality of \eqref{eq:aux} gives
\[
 b(y)\leq\max\{b(y),S\}
 \leq b(\bar y)
 \left(\frac{\rho(\bar y)}{\rho(y)}\right)^2
 e^{\beta(\Phi(\bar y)-\Phi(y))}.
\]
The right-hand side is a \(C^2\) upper support for \(b\) at \(\bar y\).
Choose a nearest point \(X_u(x_*)\) to \(X_v(\bar y)\). Since
\[
 b(y)+\log|X_v(y)-X_u(x)|\geq0,
 \qquad x\in\overline B_1,
\]
the sum of this upper support and
\(\log|X_v(y)-X_u(x)|\) is nonnegative near \((x_*,\bar y)\) and
vanishes there. Proposition \ref{prop:distance-jacobi} therefore
applies to the support.
Hereafter, \(b_i\) and \(\Delta_{T_v}b\) denote its derivatives at
\(\bar y\). Replacing \(b\) locally by this support, by construction,
\[
 2\log\rho+\beta\Phi+\log b
\]
is constant near \(\bar y\).

From now on, all quantities are evaluated at \(\bar y\), where we
again choose a principal frame. Differentiation and Lemma~\ref{lem:phi} give
\begin{equation}\label{eq:first-order}
\begin{aligned}
 0&=-\frac{4R_i}{\rho}+\beta\Phi_i+\frac{b_i}b,\\
 \frac{b_i}b
 &=\left(\frac4\rho-\frac{\beta\mathcal G(Dv)}{|z|}\right)R_i
 +\beta \kappa_iZ_i.
\end{aligned}
\end{equation}
Differentiating once more yields
\[
\begin{aligned}
 &(2\log\rho+\beta\Phi+\log b)_{ij}\\
 &\qquad=-\frac4\rho\left(\delta_{ij}
 -\langle e_i,E_{n+1}\rangle\langle e_j,E_{n+1}\rangle-h_{ij}\langle(z,0),N_v\rangle\right)
 -\frac{8R_iR_j}{\rho^2}
 +\beta\Phi_{ij}+(\log b)_{ij}.
\end{aligned}
\]
The cutoff terms satisfy
\[
\begin{aligned}
 \Delta_{T_v}(2\log\rho)
 ={}&-\frac4\rho\sum_i(H_v-\kappa_i)\bigl(1-\langle e_i,E_{n+1}\rangle^2\bigr)
 +\frac{8\lambda_v}\rho\langle(z,0),N_v\rangle\\
 &-\frac8{\rho^2}\sum_i(H_v-\kappa_i)R_i^2
 \ge-\frac{C(n,K,\lambda_v)H_v}{\rho^2}.
\end{aligned}
\]
Since $b>S\ge4(n-1)$, Proposition \ref{prop:distance-jacobi} gives
\[
\begin{aligned}
 \Delta_{T_v}\log b
 &\ge\left(\frac b{2(n-1)}-1\right)
       \sum_i(H_v-\kappa_i)\left(\frac{b_i}b\right)^2\\
 &\ge\frac b{4(n-1)}\sum_i(H_v-\kappa_i)\left(\frac{b_i}b\right)^2.
\end{aligned}
\]
Combining these estimates with Lemma~\ref{lem:phi} and
\eqref{eq:first-order}, and using $|z|\ge r$ and $\beta\le1$, we obtain
\begin{equation}\label{eq:master}
\begin{aligned}
 0\ge{}&-\frac{CH_v}{\rho^2}
 +\frac{\beta|z|}{8(1+K^2)}\sum_i(H_v-\kappa_i)\kappa_i^2\\
 &+\frac b{4(n-1)}\sum_i(H_v-\kappa_i)
 \left\{\left(\frac4\rho-\frac{\beta\mathcal G(Dv)}{|z|}\right)R_i
 +\beta \kappa_iZ_i\right\}^2.
\end{aligned}
\end{equation}
Here $C=C(n,K,\lambda_u,\lambda_v)\ge1$ is fixed independently
of the final choices of $\beta$ and $S$, and the term
$C\beta H_v/|z|$ has been absorbed into $CH_v/\rho^2$.

Order the principal curvatures so that
$\kappa_1\ge\cdots\ge\kappa_n$. Since
$H_v^2=\sum_i\kappa_i^2+2\lambda_v$ and $H_v>0$,
\begin{equation}\label{eq:spectral-bounds}
 \kappa_1\ge\frac{H_v}{n},\qquad
 H_v-\kappa_i\geq\left(1-\frac{\sqrt{2}}{2}\right)H_v>\frac{H_v}{4}\quad(i\ge2),\qquad
 H_v-\kappa_1\ge\frac{\lambda_v}{H_v},
\end{equation}
here the last inequality follows from
\[
 H_v(H_v-\kappa_1)-\lambda_v
 =\frac12\left[(H_v-\kappa_1)^2+\sum_{i=2}^n\kappa_i^2\right]\ge0.
\]

We now fix $\beta>0$ sufficiently small and then $S$ sufficiently
large so that \eqref{eq:preliminary-parameters} and
\begin{equation}\label{eq:parameter-choice}
 \frac{4(n-1)}S+32(1+K^2)^3\beta
 \le\frac{r^2}{400C(1+K^2)^4}
\end{equation}
hold. By \eqref{eq:preliminary-parameters} and $|z|\geq r$, we have
\begin{equation}\label{eq:radial-coefficient}
 \frac2\rho\le
 \frac4\rho-\frac{\beta\mathcal G(Dv)}{|z|}
 \le\frac4\rho.
\end{equation}
We now divide the remainder of the proof into two cases.

\emph{Case 1.} Suppose that
\[
 \sum_{i=2}^nR_i^2\ge\frac{|z|^2}{100(1+K^2)^4}.
\]
Completion of the square in $\kappa_i$ gives
\begin{equation}\label{eq:completed-square}
\begin{aligned}
 &\frac{\beta|z|}{8(1+K^2)}\kappa_i^2
 +\frac b{4(n-1)}
 \left\{\left(\frac4\rho-\frac{\beta\mathcal G(Dv)}{|z|}\right)R_i
 +\beta\kappa_iZ_i\right\}^2\\
 &\quad\ge
 \left[\frac{4(n-1)}b+\frac{8\beta(1+K^2)Z_i^2}{|z|}\right]^{-1}
 \left(\frac4\rho-\frac{\beta\mathcal G(Dv)}{|z|}\right)^2R_i^2\\
 &\quad\ge
 \frac{4R_i^2}{\rho^2}
 \left[\frac{4(n-1)}b+32(1+K^2)^3\beta|z|\right]^{-1}
 \ge\frac{1600C(1+K^2)^4}{r^2\rho^2}R_i^2.
\end{aligned}
\end{equation}
Here we have used Lemma~\ref{lem:phi}, \eqref{eq:radial-coefficient},
$b>S$, $|z|<1$, and \eqref{eq:parameter-choice}.
Multiplying by $H_v-\kappa_i$ and summing over $i\ge2$, the positive
terms in \eqref{eq:master} are at least
\[
 \frac{400C(1+K^2)^4H_v}{r^2\rho^2}\sum_{i=2}^nR_i^2
 \ge\frac{4CH_v}{\rho^2},
\]
which is a contradiction.

\emph{Case 2.} It remains to consider
\[
 \sum_{i=2}^nR_i^2<\frac{|z|^2}{100(1+K^2)^4}.
\]
In this case Lemma~\ref{lem:phi} gives
\[
 R_1^2=\sum_iR_i^2-\sum_{i=2}^nR_i^2
 \ge\frac{|z|^2}{1+K^2}-\frac{|z|^2}{100(1+K^2)^4}
 \ge\frac{99|z|^2}{100(1+K^2)}>0.
\]
Using Lemma~\ref{lem:phi} again together with the
Cauchy--Schwarz inequality, we obtain
\begin{equation}\label{eq:same-sign}
\begin{aligned}
 R_1Z_1
 &=\sum_iR_iZ_i-\sum_{i=2}^nR_iZ_i\\
 &\ge\frac{|z|^2}{\sqrt{1+K^2}+K}
 -\frac{|z|}{10(1+K^2)^2}\,2(1+K^2)|z|\\
 &\ge\frac{|z|^2}{4(1+K^2)}>0.
\end{aligned}
\end{equation}
Since $|R_1|\le|z|$, we have
\[
 |Z_1|\ge\frac{|z|}{4(1+K^2)}.
\]
By \eqref{eq:radial-coefficient}, \eqref{eq:same-sign}, and
$\kappa_1>0$, we have
\[
\begin{aligned}
 &\frac b{4(n-1)}(H_v-\kappa_1)
 \left\{\left(\frac4\rho-\frac{\beta\mathcal G(Dv)}{|z|}\right)R_1
 + \beta \kappa_1Z_1\right\}^2\\
 &\qquad\ge\frac {\beta^2 b}{4(n-1)}(H_v-\kappa_1)\kappa_1^2Z_1^2\\
 &\qquad\ge\frac{b\lambda_v\beta^2|z|^2H_v}
 {64n^2(n-1)(1+K^2)^2}.
\end{aligned}
\]
Here \eqref{eq:spectral-bounds} gives
$(H_v-\kappa_1)\kappa_1^2\ge\lambda_vH_v/n^2$.
Keeping only this term in \eqref{eq:master}, we conclude that
\begin{equation}\label{eq:weighted-b-bound}
 \rho(\bar y)^2b(\bar y)
 \le\frac{64n^2(n-1)(1+K^2)^2C}
 {\lambda_v\beta^2|\bar y-y_0|^2}
 \le\frac{64n^2(n-1)(1+K^2)^2C}{\lambda_v\beta^2r^2}.
\end{equation}

Finally, whether $b(\bar y)\le S$ or $b(\bar y)>S$, the maximum
property and \eqref{eq:weighted-b-bound} give
\[
 \rho(y)^2e^{\beta\Phi(y)}\max\{b(y),S\}
 \le\rho(\bar y)^2e^{\beta\Phi(\bar y)}
       \max\{b(\bar y),S\}\le C.
\]
For $y\in B_{1/2}$,
\[
 \rho(y)\ge\frac49-\left(\frac12+\frac18\right)^2
 =\frac{31}{576},\qquad |\Phi(y)|\le C(K).
\]
Therefore $b\le C$ on $B_{1/2}$, and
\[
 v-u\ge d=e^{-b}\ge e^{-C}\quad\hbox{on }B_{1/2}.
\]
This proves the theorem.
\end{proof}

\section{The Pogorelov estimate and the interior estimate}
\label{sec:pogorelov}
In this section, we establish a comparison Pogorelov estimate for the
\(\sigma_2\)-curvature equation. We then combine this estimate with
the uniform separation obtained in Section~\ref{sec:separation} and
a boundary-decay argument to prove the interior curvature estimate
stated in Theorem~\ref{thm:main}.

Throughout this section, we write
\begin{equation*}
H:=H[u],\qquad b:=\log H.
\end{equation*}

\subsection{The mean-curvature Jacobi-type inequality}\label{subsec:jacobi}
Our starting point is the projection argument of Shankar--Yuan
\cite{ShankarYuan}. Rather than imposing a dynamic semiconvexity
condition that makes every gradient coefficient nonnegative, we keep
the directional coefficients \(\Theta_i\) explicit and classify the
indices according to whether their contributions are favorable. We
show that every unfavorable direction satisfies
\(\lvert\kappa_i\rvert>H/8\), which allows its negative gradient
contribution to be controlled later by the curvature-square term
\(T_u^{ij}h_i^kh_{kj}\).
\begin{proposition}\label{prop:jacobi}
Let $n\geq3$. Suppose that $u$ satisfies \eqref{eq:intro-equation} and $H\geq2$. Then at a point where $h_i^j$ is diagonal, the following Jacobi-type inequality holds:
\begin{equation}\label{eq:jacobi}
\Delta_{T_u}b\geq 2H-\frac{4}{H}-T_u^{ij}h_i^kh_{kj}+\sum_{i=1}^n\Theta_iT_u^{ii}b_i^2,
\end{equation}
where
\begin{equation}\label{def:Thetai}
\Theta_i:=\frac{H}{H-\kappa_i}
 \left[
 \frac{3\bigl((n+1)H^2-4H\kappa_i+2\kappa_i^2-2\bigr)}
 {(n+1)H^2-4H\kappa_i+2n\kappa_i^2-2(n+2)}-1
 \right]-1.
\end{equation}
\end{proposition}
\begin{proof}
All calculations are made at the chosen point.  Differentiating the equation
$\sigma_2(\kappa)=1$ once and twice in the direction $e_p$ gives
\begin{equation}\label{eq:diff-eq}
T_u^{ii}h_{iip}=0,\qquad T_u^{ii}h_{iipp}
 =\sum_{i,j}h_{ijp}^2-H_p^2.
\end{equation}
For fixed \(i,p\), the Codazzi and Gauss equations give
\begin{equation*}
 h_{ppii}
 =h_{iipp}+\kappa_i\kappa_p(\kappa_p-\kappa_i).
\end{equation*}
It follows that
\begin{equation}\label{eq:Delta-H}
 \Delta_{T_u}H=\sum_{p,i,j}h_{ijp}^2-|\nabla H|^2
   +2H^2-4-HT_u^{ij}h_i^kh_{kj}.
\end{equation}
Therefore,
\begin{align}
&\Delta_{T_u}b-2H+\frac{4}{H}+T_u^{ij}h_i^kh_{kj}-\sum_{i=1}^n\Theta_iT_u^{ii}b_i^2\notag\\
=&6H^{-1}\sum_{i<j<p}h_{ijp}^2+3H^{-1}\sum_{i,p}h_{iip}^2-2H^{-1}\sum_p h_{ppp}^2\notag\\
&-H^{-1}\sum_p H_p^2-H^{-2}\sum_p (1+\Theta_p)T_u^{pp}H_p^2.\label{eq:hijp}
\end{align}

We now estimate the above third-order terms by the projection argument of
Shankar--Yuan \cite{ShankarYuan}.  Fix \(p\) and denote
\[
\begin{aligned}
 \mathbf t&=(h_{11p},\ldots,h_{nnp}),
 &\mathbf f&=(T_u^{11},\ldots,T_u^{nn}),\\
 \mathbf 1&=(1,\ldots,1),
 &E_p&=(0,\ldots,0,\underbrace{1}_{p\text{-th}},0,\ldots,0).
\end{aligned}
\]
The first identity in \eqref{eq:diff-eq} shows that
\(\mathbf t\cdot\mathbf f=0\).  Since
\(\sigma_2(\kappa)=1\), we also have
\begin{equation*}
 |\mathbf f|^2=(n-1)H^2-2,
 \qquad
 \mathbf 1\cdot\mathbf f=(n-1)H.
\end{equation*}
Project $E_p$ and $\mathbf 1$ onto $\mathbf f^\perp$:
\begin{equation*}
 E=E_p-\frac{T_u^{pp}}{(n-1)H^2-2}\mathbf f,
 \qquad
 L=\mathbf 1-\frac{(n-1)H}{(n-1)H^2-2}\mathbf f.
\end{equation*}
Then
\begin{equation*}
 h_{ppp}=E\cdot \mathbf t,
 \qquad
 H_p=L\cdot \mathbf t,
\end{equation*}
and a direct calculation yields
\begin{equation*}
 |E|^2=1-\frac{(T_u^{pp})^2}{(n-1)H^2-2},\qquad |L|^2=1-\frac{2(n-1)}{(n-1)H^2-2},\qquad E\cdot L=1-\frac{(n-1)HT_u^{pp}}{(n-1)H^2-2}.
\end{equation*}
For a real number \(\mu\), consider on \(\mathbf f^\perp\) the
quadratic form
\[
 3|\mathbf t|^2-2(E\cdot\mathbf t)^2
 -\mu(L\cdot\mathbf t)^2.
\]
Since \(E\) is the orthogonal projection of \(E_p\) onto
\(\mathbf f^\perp\), we have \(|E|\leq1\), and hence
\[
 3I-2E\otimes E\geq I.
\]
As \(\mu\geq0\) increases, the rank-one perturbation
\(3I-2E\otimes E-\mu L\otimes L\) remains positive semidefinite
until its smallest eigenvalue first reaches zero.  The corresponding
value is the unique positive zero of its determinant. If \(E\) and \(L\) are linearly independent, direct calculation yields
\begin{align}
 &\det\!\left(
  3I-2E\otimes E-\mu L\otimes L
  \right)\big|_{\operatorname{span}\{E,L\}}\notag\\
 &\quad=
 \frac{
  3\bigl((n+1)H^2-4H\kappa_p+2\kappa_p^2-2\bigr)
  -\mu\bigl((n+1)H^2-4H\kappa_p
             +2n\kappa_p^2-2(n+2)\bigr)}
 {(n-1)H^2-2}\label{eq:det}.
\end{align}
If \(E\) and \(L\) are linearly dependent, write \(L=tE\).  Then
\[
\det\!\left(
 \left.
 \bigl(3I-2E\otimes E-\mu L\otimes L\bigr)
 \right|_{\operatorname{span}\{E\}}
\right)
=
3-(2+\mu t^2)|E|^2.
\]
The right-hand side of \eqref{eq:det} is three times this
one-dimensional determinant:
\[
3\bigl[3-(2+\mu t^2)|E|^2\bigr].
\]
Hence the two expressions have the same zero in \(\mu\).

We next show that the two expressions appearing in the numerator of
\eqref{eq:det} are positive.
\[
 (n+1)H^2-4H\kappa_p+2\kappa_p^2-2
 =
 2(H-\kappa_p)^2+(n-1)H^2-2>0,
\]
and
\begin{align*}
 &(n+1)H^2-4H\kappa_p+2n\kappa_p^2-2(n+2)\\
 &\quad=
 2n\left(\kappa_p-\frac Hn\right)^2
 +\left(n+1-\frac2n\right)H^2-2(n+2)\\
 &\quad\geq2n-\frac8n>0.
\end{align*}
Thus the critical value is
\begin{equation}\label{eq:mu-p}
 \mu_p=
 \frac{3\bigl((n+1)H^2-4H\kappa_p+2\kappa_p^2-2\bigr)}
 {(n+1)H^2-4H\kappa_p+2n\kappa_p^2-2(n+2)}>0.
\end{equation}
Consequently,
\[
 3|\mathbf t|^2-2(E\cdot\mathbf t)^2
 -\mu_p(L\cdot\mathbf t)^2\geq0.
\]
Equivalently, \eqref{def:Thetai} can be written as
\[
 \mu_p=1+\frac{T_u^{pp}}H(1+\Theta_p).
\]
Thus the terms in \eqref{eq:hijp} corresponding to each fixed \(p\)
are nonnegative.  Summing over \(p\), together with the remaining
nonnegative terms, proves \eqref{eq:jacobi}.

\end{proof}

Set
\begin{equation}\label{def:good-bad-directions}
 I:=\left\{i:\Theta_i\geq\frac12\right\},
 \qquad
 I^c:=\left\{i:\Theta_i<\frac12\right\}.
\end{equation}
We next establish a uniform lower bound for \(\Theta_i\), together
with an improved lower bound when \(|\kappa_i|\leq H/8\).
\begin{lemma}\label{lem:weak-directions}
Suppose that $n\geq3$, $\kappa\in\Gamma_2$,
$\sigma_2(\kappa)=1$, and $H\geq2$. Then, for every $i$,
\begin{equation}\label{eq:Theta-uniform-lower}
 \Theta_i\geq-\frac{n-4}{n}.
\end{equation}
Moreover,
\begin{equation}\label{eq:Theta-central}
 |\kappa_i|\leq\frac H8
 \quad\Longrightarrow\quad
 \Theta_i\geq\frac{1}{2}.
\end{equation}
\end{lemma}
\begin{proof}
Set \(s=H-\kappa_i>0\).  The Cauchy--Schwarz inequality gives
\[
 \sigma_2(\kappa)
 \leq
 \kappa_i s+\frac{n-2}{2(n-1)}s^2.
\]
Since \(\sigma_2(\kappa)>0\),
\[
 0<2(n-1)\kappa_i+(n-2)s
 =n\kappa_i+(n-2)H.
\]
In particular, \((n-3)\kappa_i+(n+1)H>0\).
By enlarging the positive denominator and decreasing the numerator,
we obtain
\begin{align*}
\Theta_i+1=\frac{2H\left((n-3)\kappa_i+(n+1)H\right)+\frac{2(n-1)H}{H-\kappa_i}}{(n+1)H^2-4H\kappa_i+2n\kappa_i^2-2(n+2)}\geq \frac{2H\left((n-3)\kappa_i+(n+1)H\right)}{(n+1)H^2-4H\kappa_i+2n\kappa_i^2}.
\end{align*}
It follows that
\begin{align*}
 \Theta_i+1-\frac{4}{n}&\geq\frac{-8n\kappa_i^2+(2n^2-6n+16)\kappa_iH+(2n^2-2n-4)H^2}{n\left((n+1)H^2-4H\kappa_i+2n\kappa_i^2\right)}\\
 &=\frac{2\left((n-2)H+n\kappa_i\right)\left((n+1)H-4\kappa_i\right)}{n\left((n+1)H^2-4H\kappa_i+2n\kappa_i^2\right)}.
\end{align*}
The first factor is positive by the preceding estimate.  On the
other hand, since \(n\geq3\) and \(H-\kappa_i>0\),
\[
 (n+1)H-4\kappa_i
 \geq4(H-\kappa_i)>0.
\]
Thus the product of the two factors is nonnegative, and
\eqref{eq:Theta-uniform-lower} follows.

For \eqref{eq:Theta-central}, a similar calculation yields
\begin{align*}
\Theta_i+1-\frac{3}{2}\geq \frac{-3n\kappa_i^2+2n\kappa_iH+\frac{n+1}{2}H^2}{(n+1)H^2-4H\kappa_i+2n\kappa_i^2}.
\end{align*}
Since \(|\kappa_i|\leq H/8\),
\begin{align*}
 &-3n\kappa_i^2+2n\kappa_iH+\frac{n+1}{2}H^2\\
 &\quad\geq
 \left(
 -\frac{3n}{64}-\frac n4+\frac{n+1}{2}
 \right)H^2
 =
 \frac{13n+32}{64}H^2>0.
\end{align*}
This proves \eqref{eq:Theta-central}.
\end{proof}

\begin{corollary}\label{cor:classified-jacobi}
Under the assumptions of Proposition~\ref{prop:jacobi}, with
\(I\) defined by
\eqref{def:good-bad-directions}, we have
\begin{equation}\label{eq:classified-gradient}
 \sum_i\Theta_iT_u^{ii}b_i^2
 \geq
 \frac12\sum_{i\in I}T_u^{ii}b_i^2
 -\frac{n-4}{n}
  \sum_{i\notin I}T_u^{ii}b_i^2.
\end{equation}
Moreover,
\begin{equation}\label{eq:bad-carrier}
 T_u^{ij}h_i^kh_{kj}\geq\frac{H^2}{64}
 \sum_{i\notin I}T_u^{ii}.
\end{equation}
Consequently, \eqref{eq:jacobi} implies
\begin{equation}\label{eq:classified-jacobi}
 \Delta_{T_u}b
 \geq
 H-T_u^{ij}h_i^kh_{kj}
 +\frac12\sum_{i\in I}T_u^{ii}b_i^2
 -\frac{n-4}{n}
  \sum_{i\notin I}T_u^{ii}b_i^2.
\end{equation}
\end{corollary}

\begin{proof}
The estimate \eqref{eq:classified-gradient} follows directly from
\eqref{eq:Theta-uniform-lower} and the definition of \(I\). Moreover, by
\eqref{eq:Theta-central},
\begin{equation*}
 i\notin I
 \quad\Longrightarrow\quad
 |\kappa_i|>\frac H8.
\end{equation*}
Since $T_u^{ii}>0$, this gives
\begin{equation*}
 T_u^{ij}h_i^kh_{kj}=\sum_iT_u^{ii}\kappa_i^2
 \geq
 \sum_{i\notin I}T_u^{ii}\kappa_i^2
 \geq
 \frac{H^2}{64}\sum_{i\notin I}T_u^{ii},
\end{equation*}
which is \eqref{eq:bad-carrier}.  Finally,
$2H-4/H\geq H$ when $H\geq2$; hence
\eqref{eq:classified-jacobi} follows from \eqref{eq:jacobi} and
\eqref{eq:classified-gradient}.
\end{proof}

The classified Jacobi inequality
\eqref{eq:classified-jacobi} still contains the unfavorable
zeroth-order term \(-T_u^{ij}h_i^kh_{kj}\). To control this term, set
\begin{equation}
 G=\log H+\frac{n-1}{n-2}\log W.
 \label{pog:def-G}
\end{equation}
The coefficient \(\frac{n-1}{n-2}>1\) is chosen so that the coefficient
of \(T_u^{ij}h_i^kh_{kj}\) changes from \(-1\) to a positive constant.
\begin{lemma}\label{lem:angle-completed-jacobi}
Under the assumptions of Proposition~\ref{prop:jacobi}, with
\(I\) defined by
\eqref{def:good-bad-directions}, we have
\begin{equation}
 \Delta_{T_u}G
 \ge H+\frac{1}{n-2}T_u^{ij}h_i^kh_{kj}
 +\frac{n-2}{3n-5}\sum_{i\in I}T_u^{ii}G_i^2
 -c_n\sum_{i\notin I}T_u^{ii}G_i^2,
 \label{pog:angle-completed-jacobi}
\end{equation}
where
\begin{equation}
 c_n=
 \begin{cases}
  0,&n=3,4,\\[1mm]
  \dfrac{(n-4)(n-2)}{3n-4},&n\ge5.
 \end{cases}
 \label{pog:bad-loss}
\end{equation}
\end{lemma}
\begin{proof}
At the fixed
point \(X_u(p)\in\Sigma_u\), choose \(e_1,\ldots,e_n\) to be an orthonormal principal
frame of \(T_{X_u(p)}\Sigma_u\), and extend it locally as a geodesic frame. Hence
\[
 h_{ij}(p)=\kappa_i\delta_{ij},
 \qquad
 \nabla_{e_i}e_j(p)=0.
\]
Write the tangential-normal decomposition of the constant vector
\(E_{n+1}\) as
\begin{equation*}
 E_{n+1}=\sum_k \langle E_{n+1},e_k\rangle e_k+\langle E_{n+1},\nu\rangle\nu.
\end{equation*}
The Gauss and Weingarten equations give
\begin{equation}
 \nabla_i\langle E_{n+1},e_k\rangle=-\langle E_{n+1},\nu\rangle h_{ik},
 \qquad
 \nabla_i\langle E_{n+1},\nu\rangle=h_{ik}\langle E_{n+1},e_k\rangle.
 \label{pog:angle-first-derivatives}
\end{equation}
Using the Codazzi equations, we obtain
\[
 \nabla_i\nabla_j\langle E_{n+1},\nu\rangle
 =
 \langle E_{n+1},e_k\rangle\nabla_k h_{ij}
 -\langle E_{n+1},\nu\rangle h_i^k h_{kj}.
\]
It follows that
\begin{equation*}
 \Delta_{T_u}\langle E_{n+1},\nu\rangle
 =
\langle E_{n+1},e_k\rangle T_u^{ij}\nabla_kh_{ij}
 -\langle E_{n+1},\nu\rangle T_u^{ij}h_i^kh_{kj}.
\end{equation*}
Differentiating the equation
\(\sigma_2(\kappa[u])=1\) in the direction \(e_k\) yields $T_u^{ij}\nabla_kh_{ij}=0$. Consequently,
\begin{equation}
 0=\Delta_{T_u}\langle E_{n+1},\nu\rangle+\langle E_{n+1},\nu\rangle T_u^{ij}h_i^kh_{kj}.
 \label{pog:angle-jacobi}
\end{equation}
Since $ \langle E_{n+1},\nu\rangle=W^{-1}$, we have
\begin{align*}
\Delta_{T_u}\log W=
-\frac{\Delta_{T_u}\langle E_{n+1},\nu\rangle}{\langle E_{n+1},\nu\rangle}
+\frac{|\nabla\langle E_{n+1},\nu\rangle|_{T_u}^2}{\langle E_{n+1},\nu\rangle^2}=
T_u^{ij}h_i^kh_{kj}+|\nabla\log W|_{T_u}^2.
\end{align*}
Combining this identity with \eqref{eq:classified-jacobi}, we obtain
\begin{align}
 \Delta_{T_u}G
 \geq{}& H+\frac{1}{n-2}T_u^{ij}h_i^kh_{kj} \notag\\
 &+\sum_{i\in I}T_u^{ii}
 \left(\frac12b_i^2+\frac{n-1}{n-2}(\log W)_i^2\right) \notag\\
 &+\sum_{i\notin I}T_u^{ii}
 \left(-\frac{n-4}{n}b_i^2
       +\frac{n-1}{n-2}(\log W)_i^2\right).
 \label{eq:DeltaG-1}
\end{align}
Since
\[
 \frac12b_i^2+\frac{n-1}{n-2}(\log W)_i^2
 \geq\frac{n-2}{3n-5}
 \left(b_i+\frac{n-1}{n-2}(\log W)_i\right)^2
 =\frac{n-2}{3n-5}G_i^2,
\]
the directions in \(I\) give the desired positive term.  For
\(n\geq5\), the value of \(c_n\) in
\eqref{pog:bad-loss} is the smallest number for which
\[
 -\frac{n-4}{n}x^2+\frac{n-1}{n-2}y^2
 +c_n\left(x+\frac{n-1}{n-2}y\right)^2\geq0
\]
for every \(x,y\in\mathbb R\).  For \(n=3,4\), the same inequality
holds with \(c_n=0\).  Hence
\[
 -\frac{n-4}{n}b_i^2+\frac{n-1}{n-2}(\log W)_i^2
 \geq-c_n G_i^2.
\]
It follows from \eqref{eq:DeltaG-1} that \eqref{pog:angle-completed-jacobi} holds.
\end{proof}

\subsection{The two-surface Pogorelov estimate}\label{Subsec:The two-surface Pogorelov estimate}
In this subsection, we combine
Lemma~\ref{lem:angle-completed-jacobi} with a two-surface comparison
argument to prove a comparison Pogorelov estimate for the
\(\sigma_2\)-curvature equation.
\begin{theorem}\label{thm:comparison-pogorelov}
Let \(n\geq3\), let \(\Omega\subset\mathbb R^n\) be a bounded domain, and let
\[
 u\in C^4(\Omega)\cap C^2(\overline\Omega),
 \qquad
 v\in C^2(\overline\Omega).
\]
Suppose that
\begin{equation}
 \kappa[u],\kappa[v]\in\Gamma_2,
 \qquad
 \sigma_2(\kappa[u])=1,
 \qquad
 \rho:=v-u>0\quad\text{in }\Omega,
 \qquad
 \rho=0\quad\text{on }\partial\Omega,
 \label{pog:comparison-assumptions}
\end{equation}
and
\begin{equation}
 \lVert u\rVert_{C^1(\Omega)}
 +\lVert v\rVert_{C^1(\Omega)}\leq K.
 \label{pog:gradient-assumption}
\end{equation}
Then there are \(\beta=\beta(n,K)>2\) and
\(C=C(n,K)<\infty\) such that
\begin{equation}
 \sup_\Omega \rho^\beta H[u]\leq C.
 \label{pog:comparison-conclusion}
\end{equation}
\end{theorem}

\

For the remainder of this subsection, we work under the assumptions
of Theorem~\ref{thm:comparison-pogorelov}. By increasing \(K\), we may assume that \(K\geq1\). Recall
\[
 X_u(x)=(x,u(x)),\qquad X_v(y)=(y,v(y)),\qquad
 \Sigma_u=X_u(\Omega),\qquad \Sigma_v=X_v(\Omega).
\]
We first explain why a two-surface argument is needed. Let \(\beta>2\); its value will be fixed later in
terms of \(n\) and \(K\). Consider the standard
one-point test function
\[
\rho^\beta e^G
=\rho^\beta H[u]W^{\frac{n-1}{n-2}}
\]
on \(\overline\Omega\). Since \(u\in C^2(\overline\Omega)\), this
function is continuous on \(\overline\Omega\). It vanishes on
\(\partial\Omega\), while it is positive in \(\Omega\). Hence it
attains its maximum at some \(x_0\in\Omega\).

If we apply the maximum principle directly at \(x_0\), the calculation
produces the term
\[
T_u^{ij}\nabla_i\nabla_jv.
\]
Unlike in the Hessian setting, the concavity of \(\sigma_2^{1/2}\) does not directly control this term. In fact, the examples in
Remark~\ref{rem:intrinsic-one-point-obstruction} show that it can
have the wrong sign and arbitrarily large size.

We therefore introduce a two-surface test function. Set
\[
r:=\rho(x_0).
\]
Then
\begin{equation}\label{eq:r}
 0<r=\rho(x_0)
 \leq\lVert\rho\rVert_{L^\infty(\Omega)}
 \leq2K.
\end{equation}
Restrict \(x\) to the annular region
\[
\frac r2<\rho(x)<2r,
\]
which contains \(x_0\). Since \(\lvert D\rho\rvert\leq K\) and
\(\rho=0\) on \(\partial\Omega\), it follows that
\begin{equation}
 \operatorname{dist}(x,\partial\Omega)
 \geq\frac{\rho(x)}{K}
 >\frac{r}{2K}.
 \label{pog:distance-to-boundary}
\end{equation}
Thus the ball \(B_{r/(2K)}(x)\) is contained in \(\Omega\) at every
point in this annular region.

Starting from the diagonal pair \((x_0,x_0)\), we introduce a second
point \(y\) near \(x\). More precisely, define
\[
\mathcal D
=\left\{(x,y):
  \frac r2<\rho(x)<2r,
  \quad
  |y-x|<\frac{r}{2K}
 \right\}.
\]
By \eqref{pog:distance-to-boundary}, every such \(y\) lies in \(\Omega\). On \(\mathcal D\), define the two-surface auxiliary function
\begin{equation}\label{pog:def-eta}
\eta(x,y)=v(y)-u(x)-\frac r2
 -\frac r{32}-\frac{8K^2}{r}\lvert y-x\rvert^2.
\end{equation}
The quadratic term in \eqref{pog:def-eta} makes \(\eta\) negative
when \(\lvert y-x\rvert=r/(2K)\), so the positive component containing
\((x_0,x_0)\) remains inside \(\mathcal D\) (see Lemma \ref{lem:interior-contact}).
On the open set \(\mathcal D\cap\{\eta>0\}\), consider the
two-surface test function
\begin{equation}\label{def:P}
\mathcal P(x,y)=G(x)+\frac{\beta}{2}\log\eta(x,y).
\end{equation}
At the diagonal pair \((x_0,x_0)\),
\[
 \eta(x_0,x_0)=r-\frac r2-\frac r{32}=\frac{15r}{32}>0.
\]

Allowing \(x\) and \(y\) to vary independently lets us compute the
\(u\)-terms on the graph of \(u\) and the \(v\)-terms on the graph of
\(v\). In particular, \(v\) is no longer treated as a function on the
graph of \(u\). This avoids the uncontrolled one-point term, but it
also creates a geometric mismatch: the curvature tensors of the two
graphs act on different tangent spaces. We first show that \(\mathcal P\) attains its maximum
at an interior point.
\begin{lemma}\label{lem:interior-contact}
The function \(\mathcal P(x,y)\) defined in \eqref{def:P} attains its maximum at an interior point
\((x_*,y_*)\in\mathcal D\) with \(\eta(x_*,y_*)>0\).  Moreover,
\begin{equation}
 \eta(x,y)\leq \rho(x)-\frac r2\leq\frac{3r}{2},
 \label{pog:eta-comparison}
\end{equation}
and
\begin{equation}
 Dv(y_*)=\frac{16K^2}{r}(y_*-x_*),
 \qquad
 \lvert y_*-x_*\rvert
 \leq\frac{r}{16K}.
 \label{pog:y-stationarity}
\end{equation}
\end{lemma}

\begin{proof}
For \((x,y)\in\mathcal D\), the segment from \(x\) to \(y\) lies in
the ball by \eqref{pog:distance-to-boundary}. We may therefore
integrate the gradient bound for \(v\) along this segment. Using
Young's inequality, we obtain
\[
 \eta(x,y)
 \leq \rho(x)+K\lvert y-x\rvert-\frac r2
 -\frac r{32}-\frac{8K^2}{r}\lvert y-x\rvert^2
 \leq \rho(x)-\frac r2.
\]
This proves \eqref{pog:eta-comparison}; by continuity, the same estimate
holds on \(\overline{\mathcal D}\).

The closure of \(\mathcal D\) is compact in
\(\overline\Omega\times\overline\Omega\). Hence the continuous function
\[
 (x,y)\longmapsto
 e^{G(x)}\bigl(\max\{\eta(x,y),0\}\bigr)^{\beta/2}
\]
attains a maximum there, and its value is positive at \((x_0,x_0)\).
We now show that no boundary point can be a maximizer.

If \(\rho(x)=r/2\), then
\eqref{pog:eta-comparison} gives \(\eta(x,y)\leq0\).
Hence this part of the boundary contains no maximizer.

On the part where
\(\lvert y-x\rvert=r/(2K)\), the quadratic term in
\eqref{pog:def-eta} equals \(2r\).
Since \(\rho(x)\leq2r\) on the closure, the same estimate as
above gives \(\eta(x,y)\leq-r/32<0\).
If \(y\in\partial\Omega\), then
\(r/(2K)\leq\operatorname{dist}(x,\partial\Omega)\leq|y-x|\),
so this case is included in the boundary part just considered.

For $\rho(x)=2r$, the choice of \(x_0\) gives
\[
 e^{G(x)}\leq2^{-\beta}e^{G(x_0)}.
\]
Together with \eqref{pog:eta-comparison}, this yields
\[
 \frac{e^{G(x)}\eta(x,y)^{\beta/2}}
      {e^{G(x_0)}\eta(x_0,x_0)^{\beta/2}}
 \leq2^{-\beta}\left(\frac{16}{5}\right)^{\beta/2}
 =\left(\frac45\right)^{\beta/2}<1.
\]
Thus no point with \(\rho(x)=2r\) can be a maximizer.

Therefore, its maximum lies in the interior of
\(\{\eta>0\}\cap\mathcal D\).  At the maximum point,
differentiation with respect to \(y\) gives the first
identity in \eqref{pog:y-stationarity}.  The second follows from
\(\lvert Dv\rvert\leq K\).
\end{proof}

\

We now address the tangent-space mismatch described above. Let \((x_*,y_*)\) be the maximizing pair from
Lemma~\ref{lem:interior-contact}, and write
\[
 X_*:=X_u(x_*),\qquad Y_*:=X_v(y_*).
\]
We regard \(\eta\) as a function on \(\Sigma_u\times\Sigma_v\)
through the product immersion \(X_u\times X_v\). To compare the curvature contributions from the
two factors, we pair their tangent spaces using the canonical
rotation from the proof of
Proposition~\ref{prop:distance-jacobi}. The following lemma records
the resulting product-Hessian estimate needed in the
maximum-principle argument for \(\mathcal P\).

\begin{lemma}\label{lem:rotated-product-hessian}
For any orthonormal principal frame \(\{e_i\}\) at \(X_*\), there
is an orthonormal frame \(\{\widetilde e_i\}\) at \(Y_*\) such
that, with
\[
 \eta_i:=\mathrm d\eta(e_i,\widetilde e_i),
\]
we have
\begin{equation}
 \begin{aligned}
 \sum_iT_u^{ii}
 \operatorname{Hess}_{\Sigma_u\times\Sigma_v}\eta
 \bigl((e_i,\widetilde e_i),(e_i,\widetilde e_i)\bigr)\geq-C(n,K)
 -\frac{C(n,K)}r\sum_iT_u^{ii}\eta_i^2,
 \qquad |\eta_i|\leq C(K).
 \end{aligned}
 \label{pog:product-hessian-estimate}
\end{equation}
\end{lemma}
\begin{proof}
All calculations below are made at \((X_*,Y_*)\). Let \(\pi\)
be the horizontal projection and let
\(W_v=(1+|Dv(y_*)|^2)^{1/2}\). With the downward vector
\(E_{n+1}=(0,\ldots,0,-1)\), formula \eqref{pog:def-eta} becomes
\begin{equation}
 \eta(X,Y)
 =
 \langle-E_{n+1},Y-X\rangle
 -\frac r2-\frac r{32}
 -\frac{8K^2}{r}\bigl|\pi(Y-X)\bigr|^2.
 \label{pog:eta-on-product}
\end{equation}

By \eqref{pog:y-stationarity},
\begin{equation}
 E_{n+1}+\frac{16K^2}{r}(y_*-x_*,0)
 =(Dv(y_*),-1)=W_vN_v.
 \label{pog:stationary-vector}
\end{equation}
Differentiating \eqref{pog:eta-on-product} in arbitrary tangent
directions \(Z\in T_{X_*}\Sigma_u\) and
\(\widetilde Z\in T_{Y_*}\Sigma_v\) therefore gives
\begin{equation}
 \begin{aligned}
 \mathrm d\eta(Z,\widetilde Z)
 &=\left\langle
 -E_{n+1}-\frac{16K^2}{r}(y_*-x_*,0),\widetilde Z-Z
 \right\rangle\\
 &=W_v\langle N_v,Z\rangle
   -W_v\langle N_v,\widetilde Z\rangle
 =W_v\langle N_v,Z\rangle.
 \end{aligned}
 \label{pog:product-first-derivative}
\end{equation}

We use the isometry from the proof of
Proposition~\ref{prop:distance-jacobi}, applied at \((X_*,Y_*)\):
\begin{equation}
 \mathcal RZ
 =Z-\frac{\langle N_v,Z\rangle}
 {1+\langle N_u,N_v\rangle}(N_u+N_v).
 \label{pog:rotation}
\end{equation}
It maps \(T_{X_*}\Sigma_u\) isometrically onto
\(T_{Y_*}\Sigma_v\) and fixes their common tangent space.
The denominator has a uniform positive lower bound. Indeed,
\[
 \begin{aligned}
 2\bigl(1+\langle N_u,N_v\rangle\bigr)
 &=|N_u+N_v|^2\\
 &\geq\langle N_u+N_v,E_{n+1}\rangle^2
 \geq\frac4{1+K^2}.
 \end{aligned}
\]
Hence
\begin{equation}
 1+\langle N_u,N_v\rangle\geq\frac2{1+K^2}.
 \label{pog:angle-bound}
\end{equation}
For the given principal frame at \(X_*\), choose
\begin{equation}
 \widetilde e_i:=\mathcal Re_i
 \qquad\text{at }Y_*.
 \label{pog:aligned-frame}
\end{equation}
Extend the two frames independently as local orthonormal tangent
frames satisfying
\[
 \nabla^{\Sigma_u}_{e_j}e_i=0\quad\text{at }X_*,
 \qquad
 \nabla^{\Sigma_v}_{\widetilde e_j}\widetilde e_i=0
 \quad\text{at }Y_*.
\]
The frame \(\{e_i\}\) is principal at \(X_*\). In the frame
\(\{\widetilde e_i\}\), we write
\(h_i^j[v]=h_v(\widetilde e_i,\widetilde e_j)\); this matrix need not
be diagonal.

By \eqref{pog:product-first-derivative},
\begin{equation}
 \eta_i=W_v\langle N_v,e_i\rangle.
 \label{pog:paired-first-derivative}
\end{equation}
Combining \eqref{pog:rotation} and
\eqref{pog:paired-first-derivative}, we obtain
\begin{equation}
 \widetilde e_i-e_i
 =
 -\frac{\eta_i}
 {W_v\bigl(1+\langle N_u,N_v\rangle\bigr)}
 (N_u+N_v).
 \label{pog:rank-one-identity}
\end{equation}
Consequently,
\begin{align}
 \bigl|\pi(\widetilde e_i-e_i)\bigr|^2
 &\leq
 \bigl|\widetilde e_i-e_i\bigr|^2 \notag\\
 &=
 \frac{\eta_i^2
       |N_u+N_v|^2}
 {W_v^2\bigl(1+\langle N_u,N_v\rangle\bigr)^2} \notag\\
 &=
 \frac{2\eta_i^2}
 {W_v^2\bigl(1+\langle N_u,N_v\rangle\bigr)}
 \leq C(K)\eta_i^2,
 \label{pog:rank-one-bound}
\end{align}
where we have used \eqref{pog:angle-bound}.

We next differentiate the original expression
\eqref{pog:eta-on-product} twice in the paired directions.
The intrinsic connection terms vanish at the chosen points, so
\begin{equation}
 \begin{aligned}
 &\operatorname{Hess}_{\Sigma_u\times\Sigma_v}\eta
 \bigl((e_i,\widetilde e_i),(e_i,\widetilde e_i)\bigr)\\
 &\quad=
 \left\langle
 -E_{n+1}-\frac{16K^2}{r}(y_*-x_*,0),
 \overline\nabla_{\widetilde e_i}\widetilde e_i
 -\overline\nabla_{e_i}e_i
 \right\rangle
 -\frac{16K^2}{r}|\pi(\widetilde e_i-e_i)|^2.
 \end{aligned}
 \label{pog:paired-geodesic-second-derivative}
\end{equation}
By the Gauss formula,
\[
 \overline\nabla_{e_i}e_i=-\kappa_i[u]N_u,
 \qquad
 \overline\nabla_{\widetilde e_i}\widetilde e_i=-h_i^i[v]N_v.
\]
Substitution, followed by \eqref{pog:stationary-vector}, yields
\begin{equation}
 \begin{aligned}
 &\operatorname{Hess}_{\Sigma_u\times\Sigma_v}\eta
 \bigl((e_i,\widetilde e_i),(e_i,\widetilde e_i)\bigr)\\
 &\quad=
 W_v\bigl(h_i^i[v]-\langle N_u,N_v\rangle\kappa_i[u]\bigr)
 -\frac{16K^2}{r}|\pi(\widetilde e_i-e_i)|^2.
 \end{aligned}
 \label{pog:product-hessian-identity}
\end{equation}

It remains to take the \(T_u\)-trace.  The matrix \((h_i^j[v])\) is symmetric and has eigenvalues \(\kappa[v]\). Thus, the mixed
G{\aa}rding inequality gives
\begin{equation}
 \sum_iT_u^{ii}h_i^i[v]
 =
 (T_u)_i^jh_j^i[v]
 \geq
 2\sqrt{
 \sigma_2(\kappa[u])\sigma_2(\kappa[v])
 }
 >0.
 \label{pog:mixed-garding}
\end{equation}
Using also \(\sum_iT_u^{ii}\kappa_i[u]=2\), we conclude from
\eqref{pog:product-hessian-identity} and \eqref{pog:rank-one-bound}
that
\[
 \begin{aligned}
 &\sum_iT_u^{ii}
 \operatorname{Hess}_{\Sigma_u\times\Sigma_v}\eta
 \bigl((e_i,\widetilde e_i),(e_i,\widetilde e_i)\bigr)\\
 &\quad\geq
 -2W_v\langle N_u,N_v\rangle
 -\frac{16K^2}{r}\sum_iT_u^{ii}|\pi(\widetilde e_i-e_i)|^2\\
 &\quad\geq-C(n,K)-\frac{C(n,K)}r\sum_iT_u^{ii}\eta_i^2.
 \end{aligned}
\]
Finally, \eqref{pog:paired-first-derivative} gives
\[
 |\eta_i|\leq W_v\leq\sqrt{1+K^2}.
\]
This proves \eqref{pog:product-hessian-estimate}.
\end{proof}

Having located an interior maximizing point and established the
product-Hessian estimate, we now apply the maximum principle to
\(\mathcal P\).

\begin{proof}[Proof of Theorem~\ref{thm:comparison-pogorelov}]
Let \((x_*,y_*)\) be the interior maximum point of $\mathcal P$ in
$\mathcal{D}$. In the estimates below,
\(C\) denotes a positive constant depending only on \(n\) and \(K\),
independent of \(\beta\).

At the maximum of \(\mathcal P\), differentiation in the paired directions
gives
\begin{equation}
 G_i+\frac\beta{2\eta}\eta_i=0.
 \label{pog:first-variation}
\end{equation}
The second-derivative on the product manifold,
together with Lemma~\ref{lem:rotated-product-hessian}, gives
\begin{align}
 0
 &\geq\Delta_{T_u}G
 +\frac\beta{2\eta}
  \sum_iT_u^{ii}\operatorname{Hess}_{\Sigma_u\times\Sigma_v}\eta
  \bigl((e_i,\widetilde e_i),(e_i,\widetilde e_i)\bigr)
 -\frac\beta{2\eta^2}\sum_iT_u^{ii}\eta_i^2\notag\\
 &\geq\Delta_{T_u}G
 -\frac{C\beta}{\eta}
 -\frac{C\beta}{r\eta}\sum_iT_u^{ii}\eta_i^2
 -\frac\beta{2\eta^2}\sum_iT_u^{ii}\eta_i^2.
 \label{pog:raw-maximum-inequality}
\end{align}
Since \(\eta\leq3r/2\), the last two terms are bounded below by
\[
 -\frac{C\beta}{\eta^2}\sum_iT_u^{ii}\eta_i^2.
\]
If \(H(x_*)<2\), then \eqref{eq:r} and \eqref{pog:eta-comparison} give
\[
 H(x_*)\eta(x_*,y_*)\leq6K.
\]
Thus the desired contact bound already holds. We may therefore
assume \(H(x_*)\geq2\). Substituting
Lemma~\ref{lem:angle-completed-jacobi} and
\eqref{pog:first-variation} into
\eqref{pog:raw-maximum-inequality}, we obtain
\begin{align}
 0\geq H+\frac{1}{n-2}T_u^{ij}h_i^kh_{kj}-\frac{C\beta}{\eta}+\left(\frac{n-2}{4(3n-5)}\beta^2-C\beta\right)
   \frac1{\eta^2}
   \sum_{i\in I}T_u^{ii}\eta_i^2-\frac{C(\beta^2+\beta)}{\eta^2}
   \sum_{i\notin I}T_u^{ii}\eta_i^2.
 \label{pog:master-inequality}
\end{align}
Fix \(\beta=\beta(n,K)>2\) sufficiently large such that
\begin{equation}
 \frac{n-2}{4(3n-5)}\beta^2-C\beta\geq0.
 \label{pog:choice-beta}
\end{equation}
The good-direction term is therefore nonnegative.  By the
pointwise bound for \(\eta_i\) in \eqref{pog:product-hessian-estimate} and the estimate
\eqref{eq:bad-carrier},
\[
 \sum_{i\notin I}T_u^{ii}\eta_i^2
 \leq C(K)\sum_{i\notin I}T_u^{ii}
 \leq\frac{C(K)}{H^2}T_u^{ij}h_i^kh_{kj}.
\]
It follows that
\begin{equation}
 0\geq H+
 \left(
  \frac1{n-2}
  -\frac{C(n,K)(\beta^2+\beta)}{(H\eta)^2}
 \right)T_u^{ij}h_i^kh_{kj}
 -\frac{C(n,K)\beta}{\eta}.
 \label{pog:final-dichotomy}
\end{equation}
If \(H\eta\) exceeds a sufficiently large constant depending only on
\(n\) and \(K\), then the coefficient of
\(T_u^{ij}h_i^kh_{kj}\) in \eqref{pog:final-dichotomy} is at least
\(1/[2(n-2)]\), and \(C\beta/\eta\leq H/2\).  Equation
\eqref{pog:final-dichotomy} would then give
\[
 0\geq\frac H2+\frac{1}{2(n-2)}T_u^{ij}h_i^kh_{kj}>0,
\]
a contradiction. Therefore, we have proved
\begin{equation}
 H(x_*)\eta(x_*,y_*)\leq C(n,K).
 \label{pog:contact-bound}
\end{equation}

We finally return from $(x_*,y_*)$ to $x_0$. By
the maximality of \(\mathcal P\) and
\eqref{pog:contact-bound},
\begin{align*}
 e^{G(x_0)}\left(\frac{15r}{32}\right)^{\beta/2}
 &=e^{G(x_0)}\eta(x_0,x_0)^{\beta/2}\\
 &\leq e^{G(x_*)}\eta(x_*,y_*)^{\beta/2}\\
 &=\bigl(H\eta\bigr)(x_*,y_*)
   W(x_*)^{\frac{n-1}{n-2}}
   \eta(x_*,y_*)^{\beta/2-1}\\
 &\leq C(n,K)r^{\beta/2-1}.
\end{align*}
Hence
\[
 r e^{G(x_0)}\leq C(n,K).
\]
Combining the fact that \(x_0\) maximizes \(\rho^\beta e^G\) with
\eqref{eq:r}, we obtain
\[
\sup_\Omega \rho^\beta H[u]
\leq \sup_\Omega \rho^\beta e^G
= r^\beta e^{G(x_0)}
\leq C(n,K)r^{\beta-1}
\leq C(n,K).
\]
\end{proof}

\subsection{Completion of the interior estimate}\label{subsec:main-proof}

We now combine the uniform separation from
Theorem~\ref{thm:separation} with the comparison Pogorelov estimate
from Theorem~\ref{thm:comparison-pogorelov}. To apply the latter on an
interior level set of \(v-u\), we also need an upper bound for this gap
near \(\partial B_1\). We obtain such a bound by a barrier argument,
in place of the Poisson-kernel argument used by Li and Wu~\cite{LiWu}.
\begin{lemma}\label{lem:boundary-decay}
Let \(B_1\subset\mathbb R^n\), and suppose that
\[
 u,v\in C^2(B_1)\cap C^1(\overline{B_1}),
 \qquad
 \kappa[v]\in\Gamma_2,
\]
with
\[
 v\geq u\quad\text{in }B_1,
 \qquad
 v=u\quad\text{on }\partial B_1,
 \qquad
 \lVert Du\rVert_{L^\infty(B_1)}\leq K.
\]
Set
\[
 \tilde{d}(x):=\operatorname{dist}(x,\partial B_1)=1-|x|.
\]
Then
\begin{equation}
 0\leq v(x)-u(x)
 \leq (1+K)\left(\sqrt{2\,\tilde{d}(x)}+\tilde{d}(x)\right)
 \qquad\text{for every }x\in B_1.
 \label{main:boundary-decay}
\end{equation}
\end{lemma}

\begin{proof}
For any \(y\in\partial B_1\), define the barrier function
\[
  w_y(x):=u(y)+(1+K)\sqrt{2\left(1-x\cdot y\right)}.
\]
If \(z\in\partial B_1\), then
\[
 w_y(z)
 =u(y)+(1+K)|z-y|
 \geq u(z)=v(z).
\]
Moreover,
\begin{equation}
 \frac{\partial^2 w_y}{\partial x_i\partial x_j}
 =-2^{-3/2}(1+K)\frac{y_iy_j}{\left(1-x\cdot y\right)^{3/2}}
 \leq0.
 \label{main:cap-hessian}
\end{equation}

We claim that \(v\leq w_y\) in \(B_1\).  Otherwise,
\(v- w_y\) would have a positive interior maximum.  At that
point, the two functions have the same gradient and
\(D^2v\leq D^2 w_y\).  Since
\[
 H[f]
 =\frac1{\sqrt{1+|Df|^2}}
 \left(\delta^{ij}-\frac{f_i f_j}{1+|Df|^2}\right)f_{ij},
\]
the ellipticity of the mean-curvature operator and
\eqref{main:cap-hessian} give
\[
 H[v]\leq H[ w_y]\leq0.
\]
This contradicts \(H[v]>0\). Hence \(v\leq w_y\).

For a given \(x\in B_1\), choose a nearest boundary point \(y\).
Then
\[
 |x-y|=\tilde{d}(x).
\]
Using the Lipschitz bound for \(u\), we obtain
\begin{align*}
 0\leq v(x)-u(x)
 &\leq  w_y(x)-u(x)\\
 &\leq (1+K)\sqrt{2\,\tilde{d}(x)}+K \tilde{d}(x)\\
 &\leq (1+K)\left(\sqrt{2\,\tilde{d}(x)}+\tilde{d}(x)\right),
\end{align*}
which proves \eqref{main:boundary-decay}.
\end{proof}

We now apply the preceding estimates to the comparison solution \(v\).
The boundary-decay estimate will be used to choose an interior level
set of \(v-u\), where the comparison Pogorelov estimate can be
applied.

\begin{proof}[Proof of Theorem~\ref{thm:main}]
By increasing \(K\), we may assume \(K\geq1\).
Let \(v\) be the comparison solution defined in
\eqref{v}, with
\[
 \lambda_u=1,
 \qquad
 \lambda_v=\frac1{4}.
\]
For \(\partial B_1\subset\mathbb R^n\),
\[
 \frac{n-2}{n}\sigma_2(\kappa[\partial B_1])
 =
 \frac{(n-1)(n-2)^2}{2n}
 \geq\frac13>\frac14.
\]
Thus the boundary curvature condition in
\cite[Theorem~4.1]{Ivochkina} is satisfied.  The right-hand side is
positive and independent of the height, the boundary data are smooth,
and \(u\) is a strict admissible subsolution with the same boundary
values.  Ivochkina's theorem therefore gives a smooth admissible
solution \(v\); uniqueness follows from the comparison principle.
No quantitative estimate from the existence theorem is used below.

The strong comparison principle gives \(v>u\) in \(B_1\).  Since
\(H_v>0\), the function \(v\) has no interior maximum.  Hence
\[
 \sup_{B_1}v=\sup_{\partial B_1}u\leq K,
 \qquad
 \inf_{B_1}v\geq\inf_{B_1}u\geq-K.
\]
In particular, \(\operatorname{osc}_{B_1}v\leq2K\), and Korevaar's
interior gradient estimate~\cite{Korevaar} gives
\[
 \lVert Dv\rVert_{L^\infty(B_{7/8})}\leq C(n,K).
\]
Theorem~\ref{thm:separation}, with its gradient parameter replaced by
this data-dependent bound, now gives \(c_0=c_0(n,K)\in(0,1]\) such that
\begin{equation}
 v-u\geq c_0
 \qquad\text{in }B_{1/2}.
 \label{main:interior-separation}
\end{equation}

Apply Lemma~\ref{lem:boundary-decay} in \(B_1\). Choose
\(d_0=d_0(n,K)\in(0,1/2)\) so small that
\begin{equation}
 (1+K)\bigl(\sqrt{2d_0}+d_0\bigr)<\frac{c_0}{4}.
 \label{main:collar-choice}
\end{equation}
Choose a regular value
\[
 \delta\in\left(\frac{c_0}{4},\frac{c_0}{2}\right)
\]
of \(v-u\), and let \(\Omega\) be the connected component of
\(\{v-u>\delta\}\) containing \(B_{1/2}\). Combining Lemma~\ref{lem:boundary-decay},
\eqref{main:interior-separation}, and
\eqref{main:collar-choice}, we obtain
\[
 B_{1/2}\subset\Omega,
 \qquad
 \overline\Omega
 \subset
 \left\{x\in B_1:
 \operatorname{dist}(x,\partial B_1)>d_0\right\}
 \Subset B_1.
\]
Since \(\delta\) is a regular value, \(\partial\Omega\) is smooth and
\(v-u=\delta\) on \(\partial\Omega\).  For every \(x\in\Omega\),
the ball \(B_{d_0/2}(x)\) is contained in \(B_1\).  The scaled
interior gradient estimate of Korevaar, together with
\(\operatorname{osc}_{B_1}v\leq2K\), gives
\[
 \lVert Dv\rVert_{L^\infty(\Omega)}\leq C(n,K).
\]

Set
\[
 \rho:=v-u-\delta.
\]
Then the following properties hold:
\[
 \rho>0\quad\text{in }\Omega,
 \qquad
 \rho=0\quad\text{on }\partial\Omega,
 \qquad
 \lVert u\rVert_{C^1(\Omega)}
 +\lVert v-\delta\rVert_{C^1(\Omega)}\leq C(n,K),
 \qquad
 \lVert\rho\rVert_{L^\infty(\Omega)}\leq2K,
\]
and \eqref{main:interior-separation} gives
\[
 \rho\geq\frac{c_0}{2}
 \qquad\text{in }B_{1/2}.
\]

Apply Theorem~\ref{thm:comparison-pogorelov} to \(u\) and
\(v-\delta\). We obtain, for some
\(\beta=\beta(n,K)>2\),
\[
 \sup_\Omega \rho^\beta H[u]\leq C(n,K).
\]
Since \(\rho\geq c_0/2\) on \(B_{1/2}\),
\[
 \sup_{B_{1/2}}H[u]\leq C(n,K).
\]
Finally,
\[
 |A_u|^2=H_u^2-2\sigma_2(\kappa[u])=H_u^2-2.
\]
Thus
\[
 |\kappa[u]|=|A_u|\leq C(n,K)
 \qquad\text{in }B_{1/2},
\]
which proves \eqref{eq:main-estimate}.
\end{proof}

\section{Appendix}\label{appendix}
In this appendix, we give a maximum-principle proof of a doubling estimate for the \(2\)-Hessian equation. 
\begin{theorem}\label{thm:single-maximum}
Let $n\ge2$, and suppose that $u,v\in C^3(B_1)$ satisfy
\[
 D^2u,D^2v\in\Gamma_2,\qquad
 \sigma_2(D^2u)=1,\qquad
 \sigma_2(D^2v)=\frac14,\qquad
 u<v\quad\text{in }B_1.
\]
Assume that $|Dv|\le K$ in $B_{7/8}$. Let
$y_0\in\overline B_{1/8}$ and $0<r\le1/8$. If
\[
 v-u\ge\delta>0\quad\text{in }B_r(y_0),
\]
then
\begin{equation}\label{eq:separation}
 v-u\ge
 \delta\exp\left[-\frac{C(n)(1+K^2)}{r^6}\right]
 \quad\text{in }B_{1/2}.
\end{equation}
\end{theorem}

\begin{proof}
Set
\[
 w=v-u,\qquad b=\log\frac{\delta}{w},\qquad
 F^{ij}=\Delta v\,\delta_{ij}-v_{ij},
\]
and
\[
 z=y-y_0,\qquad \rho=\frac49-|z|^2,\qquad
 \beta=\frac{100n^2}{r^2},\qquad
 g(t)=-\frac1{2\beta}
 \log\left(1+\frac{t}{12(1+K)}\right).
\]
In $B_{2/3}(y_0)\Subset B_{7/8}$, consider
\[
 \mathcal Q
 =
 2\log\rho
 +g\bigl(z\cdot Dv-v+v(y_0)\bigr)
 +\log\log\max\left\{\frac{\delta}{w},20\right\}.
\]
The auxiliary function $\mathcal Q$ has the same form as the
test function in \cite[p.~584, equation~(2.15)]{QiuHessian},
except that $g$ is chosen here so that $g'<0$.

Since
\[
 |z\cdot Dv-v+v(y_0)|\le\frac{4K}{3},
\]
on the relevant range we have
\[
 |g|\le1,\qquad
 0<-g'\le1,\qquad
 \frac1{|g'|}\le C\beta(1+K),\qquad
 g''=2\beta(g')^2.
\]
Let $\bar y$ be a maximum point of $\mathcal Q$.
If $b(\bar y)\le\log20$, the maximum property gives
\[
 \rho^2\max\{b,\log20\}\le C,
\]
which implies \eqref{eq:separation}.
We may therefore assume that $b(\bar y)>\log20$.
Since $b\le0$ in $B_r(y_0)$, we have $|\bar y-y_0|>r$.

All subsequent calculations are made at $\bar y$.
Choose orthonormal coordinates such that
\[
 D^2v=\operatorname{diag}(v_{11},\ldots,v_{nn}),
 \qquad v_{11}\ge\cdots\ge v_{nn}.
\]
The first derivative condition gives
\begin{equation}\label{eq:single-maximum-first}
 0=-\frac{4z_i}{\rho}+g'z_iv_{ii}+\frac{b_i}{b},
 \qquad
 \frac{b_i}{b}
 =z_i\left(\frac4\rho-g'v_{ii}\right).
\end{equation}
By concavity of $\sigma_2^{1/2}$,
\[
 F^{ij}w_{ij}\le-\frac12,
 \qquad
 F^{ij}b_{ij}
 \ge\frac1{2w}+F^{ij}b_i b_j.
\]
Using $F^{ij}v_{ij}=1/2$ and $F^{ij}v_{ijk}=0$, we obtain
\[
\begin{aligned}
 0\ge F^{ij}\mathcal Q_{ij}
 ={}&
 -\frac4\rho\sum_iF^{ii}
 -\frac8{\rho^2}\sum_iF^{ii}z_i^2
 +g''\sum_iF^{ii}z_i^2v_{ii}^2+\frac12g'\\
 &+\frac{F^{ij}b_{ij}}b
 -\frac{F^{ij}b_i b_j}{b^2}.
\end{aligned}
\]
Since $b>\log20>2$, $|z_i|\le1$, $\rho\le1$, and
\[
 \sum_iF^{ii}=(n-1)\Delta v
 \ge\sqrt{\frac{n(n-1)}2}\ge1,
\]
it follows that
\begin{equation}\label{eq:single-maximum-basic}
 0\ge
 -\frac{14}{\rho^2}\sum_iF^{ii}
 +g''\sum_iF^{ii}z_i^2v_{ii}^2
 +\frac1{2b}\sum_iF^{ii}b_i^2.
\end{equation}

Choose $j$ such that $|z_j|\ge r/\sqrt n$.
We use the algebraic inequalities
\[
 F^{11}v_{11}^2
 \ge\frac1{2n^2(n-1)}\sum_iF^{ii},
 \qquad
 F^{jj}\ge\frac1{4(n-1)}\sum_iF^{ii}
 \quad(j>1).
\]

\medskip
\noindent
\textit{Case 1: $j=1$.}
Since $g'<0$ and $v_{11}>0$,
\eqref{eq:single-maximum-first} and
\eqref{eq:single-maximum-basic} give
\[
\begin{aligned}
 0
 &\ge
 -\frac{14}{\rho^2}\sum_iF^{ii}
 +\frac b2F^{11}z_1^2
       \left(\frac4\rho-g'v_{11}\right)^2\\
 &\ge
 -\frac{14}{\rho^2}\sum_iF^{ii}
 +\frac{br^2(g')^2}{2n}F^{11}v_{11}^2.
\end{aligned}
\]
Consequently,
\[
 b\rho^2
 \le\frac{C(n)}{r^2(g')^2}
 \le\frac{C(n)\beta^2(1+K^2)}{r^2}
 \le\frac{C(n)(1+K^2)}{r^6}.
\]

\medskip
\noindent
\textit{Case 2: $j>1$ and
$\left|\frac4\rho-g'v_{jj}\right|\ge\frac2\rho$.}
Then
\[
\begin{aligned}
 0
 &\ge
 -\frac{14}{\rho^2}\sum_iF^{ii}
 +\frac b2F^{jj}z_j^2
       \left(\frac4\rho-g'v_{jj}\right)^2\\
 &\ge
 -\frac{14}{\rho^2}\sum_iF^{ii}
 +\frac{2br^2}{n\rho^2}F^{jj},
\end{aligned}
\]
and hence
\[
 b\le\frac{28n(n-1)}{r^2}.
\]

\medskip
\noindent
\textit{Case 3: $j>1$ and
$\left|\frac4\rho-g'v_{jj}\right|<\frac2\rho$.}
In this case $(g'v_{jj})^2\ge4/\rho^2$.
Using $g''=2\beta(g')^2$ in
\eqref{eq:single-maximum-basic}, we obtain
\[
\begin{aligned}
 0
 &\ge
 -\frac{14}{\rho^2}\sum_iF^{ii}
 +2\beta(g')^2z_j^2v_{jj}^2F^{jj}\\
 &\ge
 \left(-14+\frac{2\beta r^2}{n(n-1)}\right)
 \frac{\sum_iF^{ii}}{\rho^2}\\
 &=
 \left(-14+\frac{200n}{n-1}\right)
 \frac{\sum_iF^{ii}}{\rho^2}>0,
\end{aligned}
\]
a contradiction.

Thus, in all possible cases,
\[
 \rho(\bar y)^2\max\{b(\bar y),\log20\}
 \le\frac{C(n)(1+K^2)}{r^6}.
\]
The maximum property of $\mathcal Q$ and the boundedness of $g$
give the same bound for $\rho(y)^2\max\{b(y),\log20\}$.
For $y\in B_{1/2}$,
\[
 |y-y_0|\le\frac58,\qquad
 \rho(y)\ge\frac49-\frac{25}{64}=\frac{31}{576}.
\]
Therefore
\[
 \log\frac{\delta}{v-u}
 \le\frac{C(n)(1+K^2)}{r^6}
 \qquad\text{in }B_{1/2},
\]
which proves \eqref{eq:separation}.
\end{proof}

\section*{Acknowledgments}
The authors acknowledge support from Grant 2025YFA1017603 of the National Key
R\&D Program of China and Grant 12571227 of the National Natural Science
Foundation of China.

\section*{Declaration on the use of artificial intelligence}
In developing Sections~3 and~4, the authors used ChatGPT Sol 5.6 for exploration and computation. The authors divided
the proof into uniform separation and a Pogorelov estimate,
and proposed Qiu's maximum-principle framework for separation.
ChatGPT then found the admissible-jet counterexample in
Remark~\ref{rem:intrinsic-one-point-obstruction}, showing that
the Li--Wu comparison argument does not extend directly to the
graphical curvature equation. The authors then asked ChatGPT to explore a family of Dirichlet
problems to replace the Li--Wu comparison function. Differentiation
gave a strong Jacobi inequality, but a uniform Lipschitz bound
for the linearized solutions was not apparent. A more involved
approach based on the Shankar--Yuan method was later found.
To simplify it, the authors proposed replacing the difference
of the two functions by a distance function. This led to the
distance-function Jacobi argument in Section~3. AI-assisted
calculations suggested the compression inequality in
Lemma~\ref{lem:compression}. For the Pogorelov estimate, the authors asked ChatGPT to test
one-point auxiliary functions and differential inequalities.
The system produced preliminary calculations and candidate
counterexamples. After checking these, the authors identified
the mismatch between the two hypersurfaces' gradients and
normals as the main obstruction and proposed a geometric
two-point comparison. Further AI-assisted calculations supplied
the precise form of $\eta(x,y)$ in \eqref{pog:def-eta}.

\bibliographystyle{amsplain}
\bibliography{references}
\printauthorinformation

\end{document}